\documentclass{scrartcl}

\usepackage[latin1]{inputenc}
\usepackage[hang]{subfigure}

\usepackage{amsmath}
\usepackage{amssymb}
\usepackage{amsthm}
\usepackage{latexsym}
\usepackage{bbm}

\usepackage{tikz}
\usepackage{xspace}
\graphicspath{{images/}}
\usepackage[style=mystyle, bibstyle=mybibstyle, maxnames=7, useauthor]{biblatex}

\usepackage{hyperref}

\theoremstyle{plain}
\newtheorem{theorem}{Theorem}[section]
\newtheorem{proposition}[theorem]{Proposition}
\newtheorem{lemma}[theorem]{Lemma}
\newtheorem{corollary}[theorem]{Corollary}

\theoremstyle{definition}

\newtheorem{definition}[theorem]{Definition}

\newtheorem{example}[theorem]{Example}

\newcommand{\Ie}{\textit{I.e.,}\xspace}
\newcommand{\ie}{\textit{i.e.,}\xspace}
\newcommand{\Cf}{\textit{Cf.}\ }
\newcommand{\cf}{\textit{cf.}\ }
\newcommand{\Eg}{\textit{E.g.}\xspace}
\newcommand{\eg}{\textit{e.g.}\xspace}

\newcommand{\R}{\ensuremath{\mathbb R}}		
\newcommand{\T}{\mathbb{T}}
\newcommand{\TP}{\ensuremath{\mathbb {TP}}}	

\newcommand{\eps}{\varepsilon}
\newcommand{\sg}{\sigma}
\newcommand{\mc}{\mathcal}

\newcommand{\tif}{\ensuremath{\text{if }}}
\newcommand{\other}{\ensuremath{\text{otherwise}}}
\newcommand{\tiff}{if and only if\xspace}
\newcommand{\twlog}{without loss of generality\xspace}

\newcommand{\tom}{tropical oriented matroid\xspace}
\newcommand{\toms}{tropical oriented matroids\xspace}
\newcommand{\tphp}{tropical pseudohyperplane\xspace}
\newcommand{\tphps}{tropical pseudohyperplanes\xspace}

\newcommand{\thp}{tropical hyperplane\xspace}
\newcommand{\thps}{tropical hyperplanes\xspace}

\newcommand{\thpas}{\thp arrangements\xspace}

\newcommand{\atphp}{arrangement of \tphps}
\newcommand{\tphpa}{\tphp arrangement\xspace}
\newcommand{\tphpas}{\tphp arrangements\xspace}
\newcommand{\phps}{pseudohyperplanes\xspace}

\newcommand{\mixsd}{mixed subdivision\xspace}		
\newcommand{\mixsds}{mixed subdivisions\xspace}	

\newcommand{\ndmixsd}{\mixsd of $\dilsimp n{d-1}$\xspace}
\newcommand{\dnmixsd}{\ndmixsd}
\newcommand{\dnmixsds}{\mixsds of $\dilsimp n{d-1}$\xspace}

\newcommand{\defn}[2][!*!,!]{\emph{#2}}

\newcommand{\homeo}{\simeq}					
\newcommand{\PLhom}{\overset{\text{\tiny\rm PL}}{\homeo}}	
 
\newcommand{\coloneq}{\mathrel{\mathop:}=}
\newcommand{\eqcolon}{=\mathrel{\mathop:}}

\newcommand{\simplex}{\triangle}			
\newcommand{\nsimplex}[1]{\simplex^{#1}}	

\DeclareMathOperator{\aff}{aff}				
\DeclareMathOperator{\lin}{lin}				
\DeclareMathOperator{\Sym}{Sym}

\newcommand{\nin}{\not\in}
\newcommand{\inv}[1]{#1^{-1}}				
\DeclareMathOperator{\bnd}{\partial\!}		
\newcommand{\compl}[1]{\overline{#1}}		
\newcommand{\cl}{\overline}					
\DeclareMathOperator{\st}{st}				
\DeclareMathOperator{\lk}{lk}				
\DeclareMathOperator{\card}{\#\!}			
\DeclareMathOperator{\rank}{rank}			
\DeclareMathOperator{\dist}{dist}			
\newcommand{\transpose}[1]{#1^{\mathrm T}}

\newcommand{\BIGOP}[1]{\mathop{\mathchoice%
{\raise-0.22em\hbox{\huge $#1$}}%
{\raise-0.05em\hbox{\Large $#1$}}{\hbox{\large $#1$}}{#1}}}

\newcommand{\disjoint}{\,\dotcup\,}

\newcommand{\dotcup}{\ensuremath{\mathaccent\cdot\cup}}	

\newcommand{\deletion}[2]{#1_{\setminus #2}}	
\newcommand{\contraction}[2]{#1_{/#2}}			

\newcommand{\restr}[2]{#1|_{#2}}				
\newcommand{\compgrop}{%
	\mathrel{\vcenter{\offinterlineskip
	\hbox{C\,}\vskip-1ex\hbox{\,\,G}}}}
\newcommand{\compgr}[2]{\compgrop_{#1,#2}}		
\DeclareMathOperator{\contypeop}{App}
\newcommand{\contypes}[2]{\contypeop_{#1,#2}}	
\newcommand{\shifths}[2]{H_{#1,#2}^+}			

\newcommand{\dilsimp}[2]{{#1}\nsimplex{{#2}}}		
\newcommand{\simpprod}[2]{\nsimplex{#1}\times\nsimplex{#2}}	

\definecolor{tud1b}{RGB}{0,94,168}
\definecolor{tud9c}{rgb}{0.72266,0.05859,0.13281}
\definecolor{tud7b}{RGB}{245,163,0}
\definecolor{tud3a}{RGB}{0,136,119}

\newcommand{\tudOneB}[1]{\textcolor{tud1b}{#1}}
\newcommand{\tudSevenB}[1]{\textcolor{tud7b}{#1}}
\newcommand{\tudNineC}[1]{\textcolor{tud9c}{#1}}

\begin{document}
\author{Silke Horn\\TU Darmstadt\\\url{shorn@opt.tu-darmstadt.de}}
\title{A Topological Representation Theorem\\for Tropical Oriented Matroids: Part II}

\maketitle

\begin{abstract}
Tropical oriented matroids were defined by Ardila and Develin in 2007. They are a tropical analogue of classical oriented matroids in the sense that they encode the properties of the types of points in an arrangement of tropical hyperplanes -- in much the same way as the covectors of (classical) oriented matroids describe the types in arrangements of linear hyperplanes.

Ardila and Develin proved that tropical oriented matroids can be represented as mixed subdivisions of dilated simplices. In this paper we show that this correspondence is a bijection.
Moreover, a tropical analogue for the Topological Representation Theorem for (classical) oriented matroids by Folkman and Lawrence is presented.

\end{abstract}

\section{Introduction} 

Oriented matroids abstract the combinatorial properties of arrangements of real hyperplanes and are ubiquitous in combinatorics. In fact, an arrangement of $n$ (oriented) real hyperplanes in $\R^d$ induces a regular cell decomposition of $\R^d$. Then the covectors of the associated oriented matroid encode the position of the points of $\R^d$ (respectively, the cells in the subdivision) relative to the each of the hyperplanes in the arrangement. It turns out though that there are oriented matroids which cannot be realised by any arrangement of hyperplanes.
The famous Topological Representation Theorem by Folkman and Lawrence \cite{Folkman/Lawrence} (see also \cite{BLSWZ}), however, states that every oriented matroid can  be realised as an arrangement of PL-\emph{pseudo}hyperplanes.

In this paper, we will study \emph{tropical} analogues of oriented matroids.

\bigskip
Tropical geometry is a by now well established subject, see \eg \cite{Ardila/Billey, Ardila/Klivans, Develin/Sturmfels, Mikhalkin06}. It is concerned with the algebraic geometry over the tropical semiring $(\bar\R\coloneq\R\cup\{\infty\},\oplus,\otimes)$, where $\oplus:\bar\R\times\bar\R\to\bar\R:a\oplus b\coloneq \min\{a,b\}$ and $\otimes:\bar\R\times\bar\R\to\bar\R: a\otimes b\coloneq a+b$ are the tropical addition and multiplication. It can be thought of as the image of a field of formal Puisseux series under the valuation map which takes a power series to its smallest exponent.

\medskip
From the combinatorial point of view though a \thp in $\T^{d-1}$ is just the (codimension-$1$-skeleton of the) polar fan of the $(d-1)$-dimensional simplex $\nsimplex{d-1}$. 
For a $(d-2)$-dimensional tropical hyperplane $H$ the $d$ connected components of $\TP^{d-1}\setminus H$ are called the \defn{(open) sectors} of $H$.

An arrangement of $n$ \thps in $\T^{d-1}$ induces a cell decomposition of $\T^{d-1}$ and each cell can be assigned a \defn{type} that describes its position relative to each of the \thps. To be precise, the point $p$ is assigned the type $A=(A_1,\ldots,A_n)$ where $A_i$ denotes the set of closed sectors of the $i$-th \thp in which $p$ is contained. See Figure \ref{fig:thpa} for an illustration in dimension $2$.

\medskip
It turns out that tropical curves -- and as such in particular arrangements of tropical hyperplanes -- have relationships to other interesting objects. Triangulations of products of two simplices are ubiquitous and utile objects in discrete geometry due to their connection with toric Hilbert schemes \cite{Santos05} and Schubert calculus \cite{Ardila/Billey} among others.

By Develin and Sturmfels \cite{Develin/Sturmfels} \emph{regular} subdivisions of $\simpprod{n-1}{d-1}$ are dual to arrangements of $n$ \thps in $\T^{d-1}$. See Figure \ref{fig:mixsd2thp} for an illustration.

\medskip
A central concept is that of an $(n,d)$-type.
\begin{definition}
\label{def:nd_type}
For $n,d\geq1$ an \defn[(n,d)-type]{$(n,d)$-type} is an $n$-tuple $(A_1,\ldots,A_n)$ of non-empty subsets of $[d]$.

For convenience we will write sets like $\{1,2,3\}$ as $123$ throughout this article.
\medskip

An $(n,d)$-type $A$ can be represented as a subgraph $K_A$ of the complete bipartite graph $K_{n,d}$: Denote the vertices of $K_{n,d}$ by $N_1,\ldots,N_n,D_1,\ldots,D_d$. Then the edges of $K_A$ are $\{\{N_i,D_j\}\mid j\in A_i\}$.
\end{definition}

Besides \thp arrangements there are other objects that share the notion of an $(n,d)$-type:
\begin{itemize}

\item If we label the vertices of $\nsimplex {n-1}$ by $1,\ldots, n$, the vertices of the polytope $\simpprod{n-1}{d-1}$ are in canonical bijection with the edges of the complete bipartite graph $K_{n,d}$. Then a cell $C$ in a subdivision of $\simpprod{n-1}{d-1}$ is assigned the type corresponding to the subgraph of $K_{n,d}$ containing all edges that mark vertices of $C$. 
See \eg De Loera, Rambau and Santos \cite{LRS-book} for a thorough treatment of this matter.

\item Given a \dnmixsd, every cell is a Minkowski sum of $n$ faces of $\nsimplex {d-1}$. By identifying the faces of $\nsimplex{d-1}$ with the subsets of $[d]$, this again yields an $(n,d)$-type. See Figure \ref{fig:mixsd2thp_a} for an example. We introduce \mixsds in  Section \ref{sec:mixsds}.

\item Tropical oriented matroids as defined by Ardila and Develin \cite{Ardila/Develin} via a set of covector axioms generalise \thpas. We define them in Section \ref{sec:toms}.
\end{itemize}

\begin{figure}[b]
\centering
\subfigure[A (regular) \mixsd of $\dilsimp22$.]
{\label{fig:mixsd2thp_a}\includegraphics[width=3.5cm]{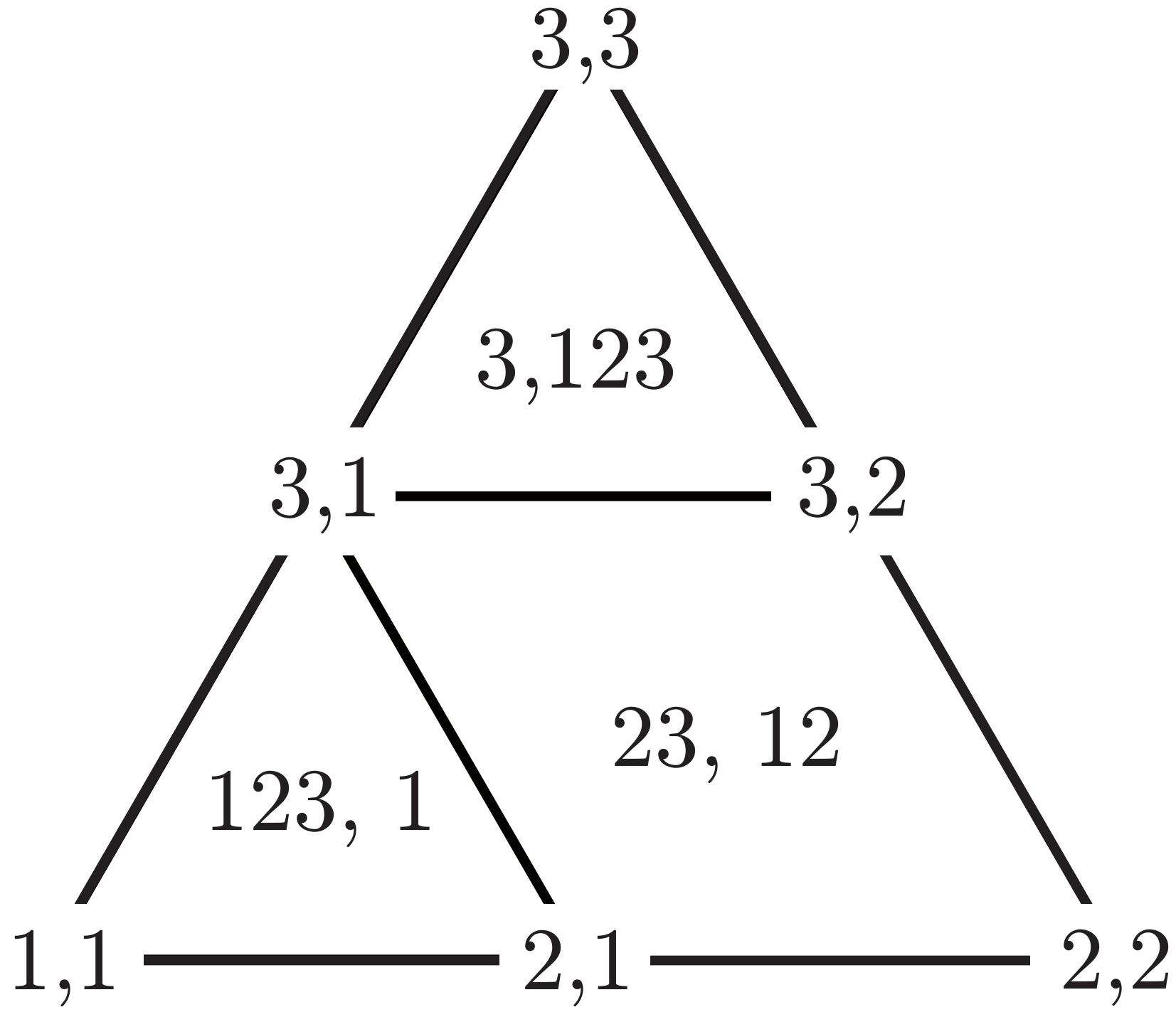}}\qquad
\subfigure[The Poincaré dual \mbox{of \subref{fig:mixsd2thp_a}}.]
{\includegraphics[width=3.5cm]{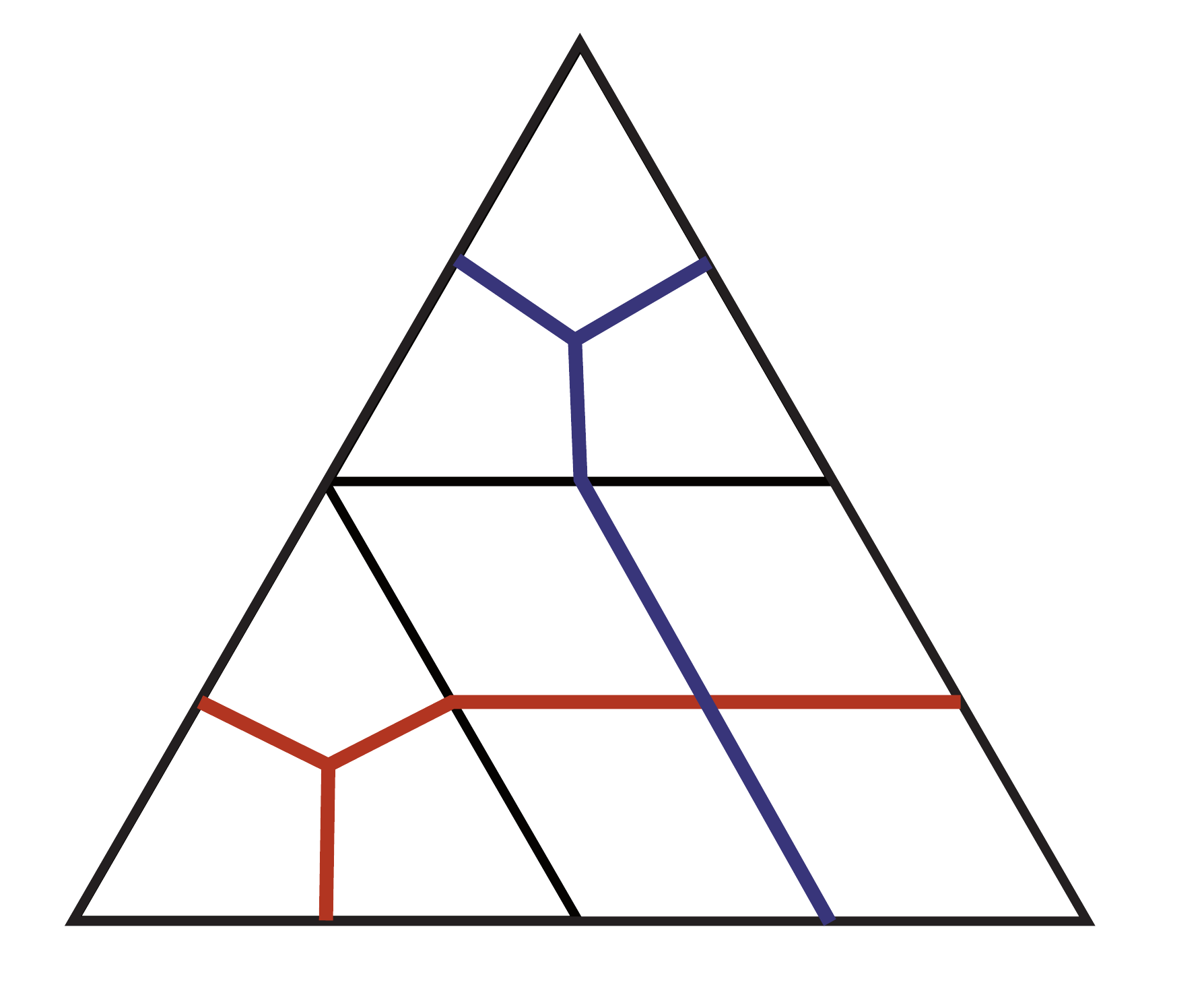}}\qquad
\subfigure[An arrangement of \thps.\label{fig:thpa}]
{\includegraphics[width=3.5cm]{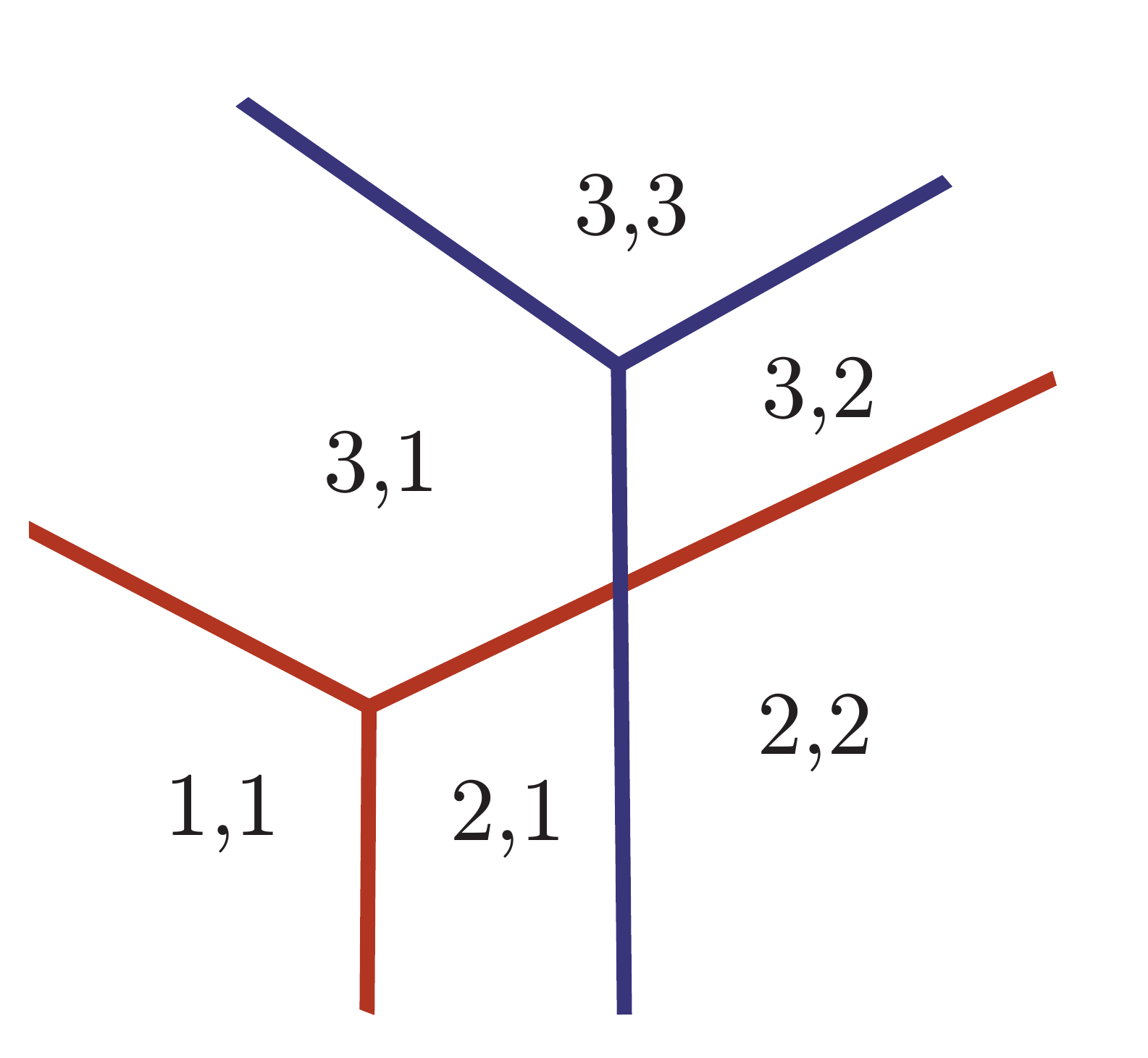}}
\caption{The correspondence between \mixsds and \tphpas.}
\label{fig:mixsd2thp}
\end{figure}

We continue by briefly pointing out what is known about the relations between the above objects.
By the Cayley Trick (\cf Huber, Rambau and Santos \cite{HRS00}) subdivisions of $\simpprod{n-1}{d-1}$ are in bijection with \mixsds of $\dilsimp n{d-1}$.

By \cite[Theorem 6.3]{Ardila/Develin}, the types of a \tom with parameters $(n,d)$  yield a subdivision of $\simpprod{n-1}{d-1}$.
They also conjecture this to be a bijection. 
By \cite[Proposition 6.4]{Ardila/Develin}, these types satisfy all but one of the \tom axioms.

In Oh and Yoo \cite{Oh/Yoo} it is proven that \emph{fine} \mixsds satisfy the elimination axiom.

Moreover, \cite{toprep1} provides further evidence for the close relationship between \dnmixsds and \toms, respectively arrangements of \thps. \Eg \cite[Theorem 4.2]{toprep1} shows that the Poincaré dual of a \dnmixsd is a family of \tphps.

\bigskip

In this paper we introduce arrangements of \tphps and prove a tropical analogue to the Topological Representation Theorem for (classical) oriented matroids by Folkman and Lawrence \cite{Folkman/Lawrence}. Another variant of the Topological Representation Theorem for a different definition of \tphpas is contained in \cite{toprep1}.

A \defn{\tphp} is basically a set which is PL-homeomorphic to a \thp (see also Definition \ref{def:tphp}). The challenging part is the definition of arrangements of these: We have to impose restrictions on the intersections of the pseudohyperplanes in the arrangement. In the classical framework, the intersections of the hyperplanes in the arrangement have to be homeomorphic to linear hyperplanes (of smaller dimension). In the tropical world, however, this approach is not feasible, since intersections of tropical hyperplanes are no longer homeomorphic to tropical hyperplanes (but have a very complicated geometry). In \cite{toprep1}, we instead imposed restrictions on the cell decomposition induced by the \tphps in the arrangement. Here we choose yet another approach that is conceptually closer to the classical case. 

An family of tropical pseudohyperplanes is an arrangement if any set of tropical halfspace boundaries forms an arrangement of affine pseudohyperplanes.


\medskip
With this definition we prove the Topological Representation Theorem:
\begin{theorem}[Topological Representation Theorem]\label{thm:int_toprep}
Every \tom (in general position) can be realised by an \atphp.
\end{theorem}

We also introduce a theory of combinatorial tropical convexity that is closely related to the elimination property of \toms. In fact, it turns out that a \dnmixsd satisfies the elimination property if and only if the combinatorial convex hull of any two cells is path-connected. Since any intersection of affine halfspaces is path-connected, we obtain the following application of Theorem \ref{thm:int_toprep}:

We show that \emph{all} \dnmixsds satisfy the elimination property and hence prove the conjecture by Ardila and Develin:

\begin{theorem}[\Cf {\cite[Conjecture 5.1]{Ardila/Develin}}]\label{thm:mainresult}
Tropical oriented matroids with parameters $(n,d)$ are in bijection with subdivisions of $\simpprod{n-1}{d-1}$ and \mixsds of $\dilsimp n{d-1}$.
\end{theorem}

\bigskip
For quick reference, the general picture is depicted in Figure \ref{fig:big_plan}.

\begin{figure}[h]
\centering
\begin{tikzpicture}
[mynode/.style={rectangle,draw=white, minimum size=6mm}]
\def\x{11}
\def\y{8}

\definecolor{linkcol}{cmyk}{1,.6,0,0}

\node (tom) at (0,0) [mynode, align=center] {tropical oriented\\matroids};
\node (product) at (\x,0) [mynode, align=center] {subdivisions of\\$\simpprod {n-1}{d-1}$};
\node (mixsd) at (\x,-\y) [mynode, align=center] {\mixsds\\of $\dilsimp n{d-1}$};
\node (pseudo) at (0,-\y) [mynode, align=center] {\tphp\\ arrangements\\{\footnotesize(\tudOneB{\cite[Def. 4.3]{toprep1}, Def. \ref{def:tphpa2}})}};

\footnotesize

\definecolor{citecol}{cmyk}{0,.4,1,0}
\draw [<->, draw=tud9c, very thick] (mixsd.north) to node[near end,auto,align=right] {\tudNineC{Cayley Trick} \\ {\definecolor{citecol}{cmyk}{.3,1,.9,0} \tudNineC{\cite[Thm.\ 3.1]{HRS00}}}} (product.south);

\draw [->, draw=tud7b, very thick] (tom.10) to node[auto] {\tudSevenB{\cite[Thm.\ 6.3]{Ardila/Develin}}} (product.170);

\draw [->, draw=tud7b, very thick, dashed] (product.190) to node[auto] {\tudSevenB{\cite[Conj.\ 5.1]{Ardila/Develin}}} (tom.352);

\draw [->, draw=tud1b, very thick] (mixsd.120) to node [midway,above,sloped, align=center] %
{\tudSevenB{\qquad~~ boundary, comparability, surrounding: \cite[Prop.\ 6.4]{Ardila/Develin}}\\ %
\tudSevenB{elimination for $d=3$: \cite[Thm.\ 6.5]{Ardila/Develin}} \\%
\definecolor{citecol}{RGB}{80,182,149}\textcolor{tud3a}{elimination for \emph{fine} case: \cite[Prop.\ 4.12]{Oh/Yoo}}
} 
node [midway,below,sloped, align=center] {\tudOneB{elimination in general case}\\\tudOneB{Thm.\ \ref{thm:elimination_nongen}}}(tom.300);

\draw [->, draw=tud1b, very thick] (mixsd.west) to node[auto,swap, align=center] {\tudOneB{Topological Representation Theorem}\\\tudOneB{\cite[Thm. 4.4]{toprep1}, Thm. \ref{thm:toprep2}}} node [auto, align=center] {\tudSevenB{\cite[Conj.\ 5.7]{Ardila/Develin}}\\ \tudNineC{realisable/regular case: \definecolor{citecol}{cmyk}{.3,1,.9,0}\cite[Thm.\ 1]{Develin/Sturmfels}}} (pseudo.east);

\draw [->, draw=tud1b, very thick] (pseudo.north) to node[near start, auto, align=left, swap] {\tudOneB{with Def. \ref{def:tphpa2}}\\\tudOneB{(general position)}\\\tudOneB{Thm.\ \ref{thm:tphpa_elim}}} (tom.south);
\end{tikzpicture}

\caption[``The bigger picture''.]{The correspondences between the four concepts of \toms, \dnmixsds, subdivisions of a product of two simplices and \tphpas.
}
\label{fig:big_plan}
\end{figure}

\bigskip
The paper is organised as follows: In Section \ref{sec:toms} we briefly review the definition of \toms. In Section \ref{sec:mixsds} we discuss \mixsds of dilated simplices. In Section \ref{sec:conv} we have a closer look at the elimination property and define a notion of convexity in \toms. In Section \ref{sec:toprep} we introduce arrangements of \tphps in analogy to (classical) pseudohyperplane arrangements (see Definition \ref{def:tphpa2}) and prove a Topological Representation Theorem (Theorem \ref{thm:toprep2}). Finally, in Section \ref{sec:elim} we apply our results to prove Theorem \ref{thm:mainresult}.

\bigskip
This is the follow-up paper of \cite{toprep1}. A joint extended abstract \cite{toprep_fpsac} of this and \cite{toprep1} has been presented at FPSAC 2012. Moreover, the results are also contained in \cite{mydiss}.

\section{Tropical Oriented Matroids}\label{sec:toms}

The following definitions are analogous to those in \cite{Ardila/Develin}, respectively \cite{toprep1}.

A \defn{refinement} of an $(n,d)$-type $A$ with respect to an ordered partition $P=(P_1,\ldots,P_k)$ of $[d]$ is the $(n,d)$-type $B=\restr A P$ where $B_i=A_i\cap P_{m(i)}$ and $m(i)$ is the smallest index where $A_i\cap P_{m(i)}$ is non-empty for each $i\in[n]$. A refinement is \defn[total!refinement]{total} if all $B_i$ are singletons.

\smallskip
Given $(n,d)$-types $A$ and $B$, the \defn{comparability graph} $\compgr AB$ is a multigraph with node set $[d]$. For $1\leq i\leq n$ there is an edge for every $j\in A_i,k\in B_i$. This edge is undirected if $j,k\in A_i\cap B_i$ and directed $j\rightarrow k$ otherwise. (We consider the comparability graph as a graph without loops.)
Note that there may be several edges (with different directions) between two nodes.

A \defn[]{directed path} in the comparability graph is a sequence $e_1,e_2,\ldots,e_k$ of incident edges at least one of which is directed and all directed edges of which are directed in the ``right'' direction. A \defn[]{directed cycle} is a directed path whose starting and ending point agree. The graph is \defn[]{acyclic} if it contains no directed cycle.

\begin{definition}[{\Cf \cite[Definition 3.5]{Ardila/Develin}}]
\label{def:trop_or_mat}
A \defn{tropical oriented matroid} $M$ (with parameters $(n,d)$) is a collection of $(n,d)$-types \index{type!in a tropical oriented matroid|see{$(n,d)$-type}} which satisfies the following four axioms:
\begin{itemize}
\item\defn[boundary axiom]{Boundary}: For each $j\in[d]$, the type $(j,j,\ldots,j)$ is in $M$.
\item\defn[comparability axiom]{Comparability}: The comparability graph $\compgr AB$ of any two types $A,B\in M$ is acyclic.
\item\defn[elimination axiom]{Elimination}: If we fix two types $A, B\in M$ and a position $j\in[n]$, then there exists a type $C$ in $M$ with $C_j=A_j\cup B_j$ and $C_k\in\{A_k,B_k,A_k\cup B_k\}$ for $k\in[n]$.
\item\defn[surrounding axiom]{Surrounding}: If $A$ is a type in $M$, then any refinement of $A$ is also in $M$.
\end{itemize}
We call $d\eqcolon\rank M$ the \defn[rank!of a \tom]{rank} and $n$ the \defn[size!of a \tom]{size} of $M$. 
\end{definition}

\begin{example}
By \cite[Theorem 3.6]{Ardila/Develin} the set of types of an arrangement of $n$ tropical hyperplanes in $\T^{d-1}$ is a \tom with parameters $(n,d)$.
\end{example}
We  call \toms coming from an arrangement of \thps \defn{realisable}. Recall that by Develin and Sturmfels \cite{Develin/Sturmfels} realisable \toms are in bijection with \emph{regular} \dnmixsds.

\begin{definition}\label{def:tom_props}
The \defn[dimension!of a type]{dimension} of an $(n,d)$-type $A$ is the number of connected components of $K_A$ minus $1$.
A \defn[vertex!in a \tom]{vertex} is a type of dimension $0$, an \defn[edge!in a \tom]{edge} a type of dimension $1$ and a \defn{tope} a type of full dimension $d-1$, \ie each tope is an $n$-tuple of singletons.

A \tom $M$ is \defn[general position]{in general position} if for every type $A\in M$ the graph $K_A$ is acyclic.

For two types $A,B$ we  write $A\supseteq B$ if $A_i\supseteq B_i$ for each $i\in[n]$. 
\end{definition}

\begin{definition}[\Cf {\cite[Propositions 4.7 and 4.8]{Ardila/Develin}}] \label{def:tom_deletion_contraction}
 Let $M$ be a \tom with parameters $(n,d)$.
\begin{enumerate}
\item For $i\in[n]$ the \defn[deletion!in a \tom]{deletion} $\deletion Mi$ consisting of all $(n-1,d)$-types which arise from types of $M$ by deleting coordinate $i$ is a \tom with parameters $(n-1,d)$.
\item For $j\in[d]$ the \defn[contraction!in a \tom]{contraction} $\contraction Mj$ consisting of all types of $M$ that do not contain $j$ in any coordinate is a \tom with parameters $(n,d-1)$.
\end{enumerate}
\end{definition}

There is also a notion of duality for $(n,d)$-types:
\begin{definition}[{\Cf\cite[Definitions 5.3 and 5.4]{Ardila/Develin}}]\label{def:tom_dual}
If $A$ is a bounded $(n,d)$-type then we get a $(d,n)$-type $\transpose A$, the \defn[dual!of an $(n,d)$-type]{dual type} of $A$, by interchanging the roles of $n$ and $d$ in the type graph $K_A$; \ie $\transpose A$ is defined by \[i\in A_j\quad \Leftrightarrow\quad j\in\transpose A_i.\]
If $M$ is a \tom with parameters $(n,d)$ then we define the \defn[dual!of a tropical oriented matroid]{dual} $\transpose M$ by \[\transpose M\coloneq\{\restr{\transpose A}P\mid A \text{ vertex of }M, p \text{ ordered partition of }[n]\}.\]
\end{definition}
We will later see in Corollary \ref{cor:dual_tom} that if $M$ is a \tom with parameters $(n,d)$, then its dual $\transpose M$ is a \tom with parameters $(d,n)$.


\section{Mixed Subdivisions of \boldmath{$\dilsimp n{d-1}$}}
\label{sec:mixsds}

Given two sets $X,Y$ their \defn{Minkowski sum} $X+Y$ is given by $X+Y\coloneq\{x+y\mid x\in X,y\in Y\}$.

\begin{definition}\label{def:mixed_sd}
Let $P_1,\ldots,P_k\subset \R^n$ be (full-dimensional) convex polytopes. Then a polytopal subdivision  $\{Q_1,\ldots,Q_s\}$ of $P\coloneq\sum P_i$ is a \defn{mixed subdivision} if it satisfies the following conditions:
\begin{enumerate}
\item Each $Q_i$ is a Minkowski sum $Q_i=\sum\limits_{j=1}^k F_{i,j}$, where $F_{i,j}$ is a face of $P_j$.
\item For $i,j\in[s]$ we have that $Q_i\cap Q_j=(F_{i,1}\cap F_{j,1})+\ldots+(F_{i,k}\cap F_{j,k})$.
\end{enumerate}
\end{definition}
A \dnmixsd is \defn{fine} if there is no other \dnmixsd refining it. 

We are interested in the case of \mixsds where $P_i=\nsimplex{d-1}$ for each $i$. Then $\sum P_i=\dilsimp n{d-1}$ is a dilated simplex. By Santos \cite{Santos03} a subdivision of $\dilsimp n{d-1}$ is mixed if and only if each cell is a Minkowski sum of $n$ faces of $\nsimplex{d-1}$. By Ardila and Develin \cite[Theorem 6.3]{Ardila/Develin} the types of a \tom with parameters $(n,d)$ yield a \dnmixsd. A \tom is in general position if and only if its \mixsd is fine.

\smallskip
To avoid confusion with the vertices of \toms, we  speak of the $0$-dimensional cells of a \mixsd as \defn{topes}. By \cite[Proposition 3.1]{toprep1}, a \dnmixsd is uniquely determined by its topes.

\subsection{Placing in mixed subdivisions}
Recall that triangulations of $\simpprod{n-1}{d-1}$ are in bijection with the fine \dnmixsds via the Cayley Trick.
There is a well-known construction that produces a triangulation of $\simpprod{n'}{d'}$ (called the \defn{placing triangulations}) from one of $\simpprod nd$ for $n'\geq n,d'\geq d$. See De Loera, Rambau and Santos \cite[Section 4.3.1]{LRS-book} for more details.

Since we will need this construction in Section \ref{sec:elim}, we now examine how placing works in the \mixsd point of view:

Suppose we are given a \mixsd $S$ of $\dilsimp n{d-1}$. Let $T$ be the corresponding subdivision of $\simpprod{n-1}{d-1}$.
There are two possible ways to extend this by placing:

\begin{itemize}
\item We can embed $T$ into $\simpprod n{d-1}$. \Ie we extend $S$ to a \mixsd of $\dilsimp {(n+1)}{d-1}$.

\item We can embed $T$ into $\simpprod{n-1}{d}$. \Ie we extend $S$ to a \mixsd of $\dilsimp {n}{d}$.
\end{itemize}

We will call the operations \defn[n-placing@$n$-placing]{$n$-placing}, respectively \defn[d-placing@$d$-placing]{$d$-placing}, referring to whether we increase $n$ or $d$.
The two operations are of course dual to each other.

\paragraph{\boldmath{$n$}-Placing}
There are $d$ vertices to be placed, namely the vertices $(n+1,1),\ldots,(n+1,d)$. We will denote both the \dnmixsd and the corresponding subdivision of $\simpprod{n-1}{d-1}$ by $S$. Morever, we will apply operations as defined for \toms to the types of both \mixsds and triangulations of products of simplices.

\medskip

Let $\sg$ be some permutation of $[d]$. 
First we place the vertex $(n+1,\sg_1)$. From this vertex every maximal (\ie $(n+d-2)$-dimensional) simplex of $S$ is visible. Thus, for every maximal simplex $B$ we add the simplex $B\cup\{\sg_1\}$ and all its faces to $S$ to get $S_1$. In the \mixsd this corresponds to adding a new entry $\{\sg_1\}$ at the end of every type in $S$. Thus, $S_1$ is just a copy of $S$ in the $\sg_1$-th corner of $\dilsimp{(n+1)}{d-1}$.
\begin{figure}[th]
\centering
\includegraphics[width=7cm]{placing.2}
\caption{A \mixsd $S$ of $\dilsimp32$ (black) in its $n$-placing extension with respect to the permutation $(1,2,3)$. 
}
\label{fig:placing_ex}
\end{figure}

As for placing the vertex $(n+1,\sg_2)$, the only visible simplices are those whose type does not contain $\sg_1$ except in the last entry (where we just added it). In the \mixsd, placing $(n+1,\sg_2)$ corresponds to appending a new entry $\{\sg_1,\sg_2\}$ to the end of every vertex in the contraction $\contraction S{\sg_1}$ and then adding all refinements of those to obtain $S_2$.

Placing the remaining vertices works similarly: When placing $(n+1, \sg_i)$, we create the set $S_i$ containing all vertices in the contraction $\contraction S {\{\sg_1, \ldots,\sg_{i-1}\}}$ with a new entry $\{\sg_1,\ldots,\sg_i\}$ appended and all refinements of those.

Figure \ref{fig:placing_ex} shows an example of an $n$-placing extension.
%

\paragraph{\boldmath{$d$}-Placing}

There are $n$ vertices to be placed, namely the vertices $(1,d+1),\ldots,(n,d+1)$. 

Let $\tau$ be some permutation of $[n]$.
Recall that for the construction of the $n$-placing extension the contractions $\contraction S{\sg_1}$, $\contraction S{\{\sg_1,\sg_2\}}$, $\ldots$, $\contraction S{\{\sg_1,\sg_2,\ldots,\sg_d\}}$ for some permutation $\sg$ of $[d]$ played an important role.

In the same way, the deletions $\deletion S{\tau_1}, \deletion S{\{\tau_1,\tau_2\}}, \ldots, \deletion S{\{\tau_1,\tau_2,\ldots,\tau_n\}}$ will be important in the construction of the $d$-placing extension of $S$.

\medskip
We will only consider the maximal simplices in $S'$.

First place the vertex $(\tau_1,d+1)$. From this vertex every maximal simplex in $S$ is visible. Hence for every maximal simplex $B$ we add the simplex $B\cup\{(\tau_1,d+1)\}$ to get $S_1$. In the \mixsd this corresponds to adding $d+1$ to $B_{\tau_1}$.

When we then place $(\tau_2,d+1)$, the visible simplices are the simplices in $S$ with the $\tau_i$-th entry replaced by $\{d+1\}$.

In general, when placing the $i$-th vertex $(\tau_i,d+1)$, the visible simplices correspond to the cells in the deletion $\deletion S{\{\tau_1,\ldots,\tau_{i-1}\}}$ with additional entries $\{d+1\}$ at the positions $\tau_1,\ldots,\tau_{i-1}$.

See Figure \ref{fig:d_placing_3d} for an illustration.

\begin{figure}[h]
\centering
\includegraphics[width=4.5cm]{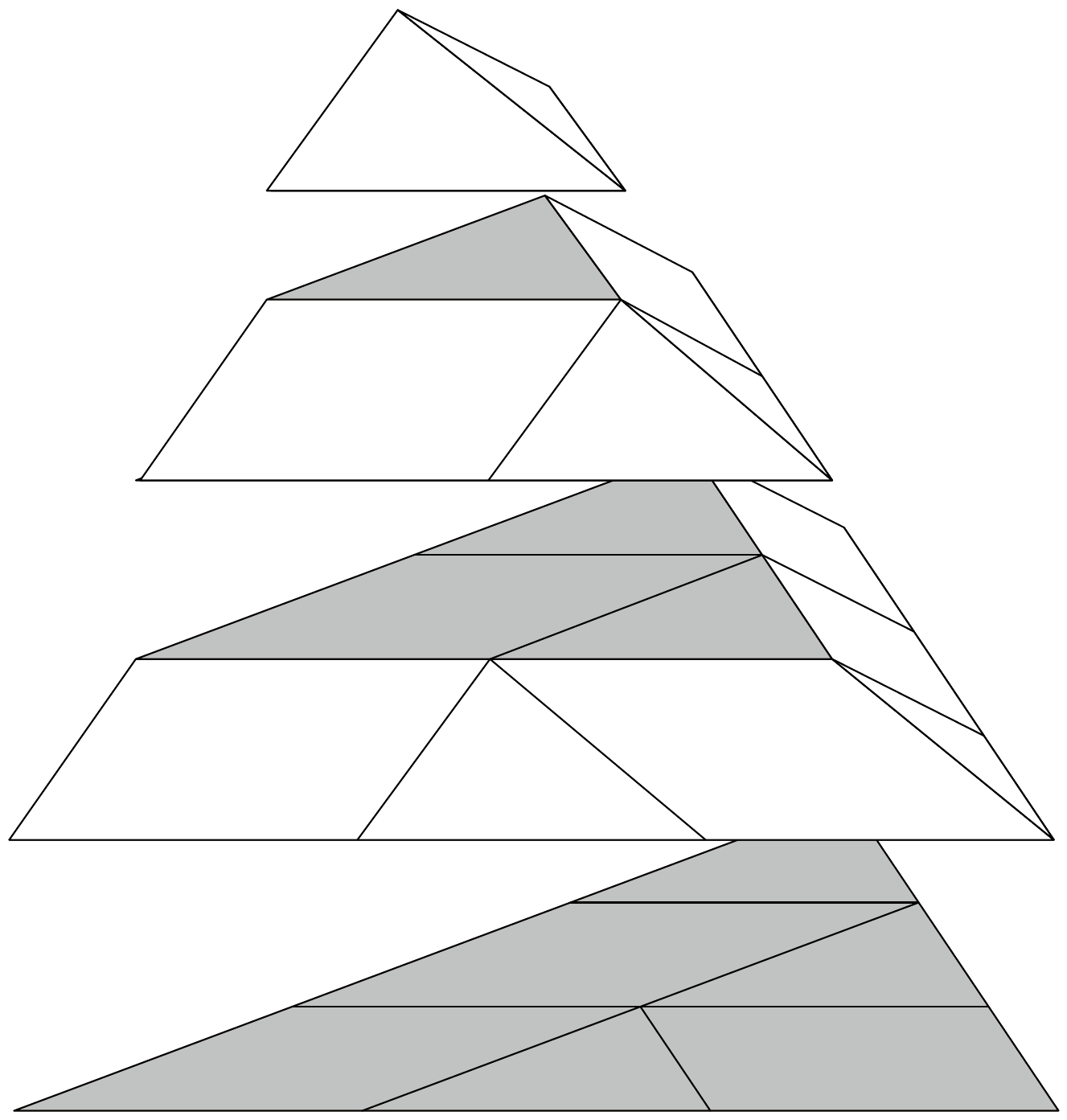}
\caption[A $d$-placing extension of a \mixsd of $\dilsimp 32$.]{A $3$-dimensional $d$-placing extension of a \mixsd of $\dilsimp32$.}
\label{fig:d_placing_3d}
\end{figure}

\section{Convexity in \toms and the elimination property}\label{sec:conv}
Recall that by Ardila and Develin \cite[Theorem 6.3]{Ardila/Develin} the types of a \tom with parameters $(n,d)$  yield a subdivision of $\simpprod{n-1}{d-1}$.
Since by \cite[Proposition 6.4]{Ardila/Develin} these types satisfy the boundary, comparability and surrounding axioms, the only thing left open is elimination.

By Oh and Yoo \cite[Proposition 4.6]{Oh/Yoo}, \emph{fine} \mixsds satisfy the elimination property.

\medskip
In the realisable case, the \emph{elimination axiom} describes the intersection of a tropical line segment from $A$ to $B$ with the $j$-th \thp. In other words, in the according arrangement of tropical \emph{pseudo}hyperplanes (dual to the \mixsd) all eliminations of $A$ and $B$ (for all $j$) describe the line segment from $A$ to $B$.

\medskip
One can exploit the elimination property of \toms to obtain topological properties of the according \mixsds.

\smallskip
\begin{definition}\label{def:tom_convex}
Let $M$ be a \tom and $A,B\in M$ two types. Then the set \[M_{AB}\coloneq\{C\in M\mid C_i\in\{A_i,B_i,A_i\cup B_i\}\text{ for all }i\in[n]\}\] is the \defn{(combinatorial) convex hull} of $A$ and $B$.
Analogously we define the \defn{(combinatorial) convex hull} $S_{AB}$ of two cells in a \mixsd $S$ of $\dilsimp n{d-1}$.

We say that a subset $C$ of a \tom $M$ (or equivalently, a subcomplex of a \dnmixsd) is \defn[convex!subset of \tom]{convex}\index{convex!subcomplex of \mixsd|see{convex!subset of \tom}} if for any $A, B\in C$ we have that $M_{AB}\subseteq C$.
\end{definition}

Develin and Sturmfels \cite{Develin/Sturmfels} defined a notion of convexity in tropical geometry: \index{tropical!convexity}
Given two points $x,y\in\T^{d-1}$ the \defn[tropical!line segment]{tropical line segment} connecting them is the set \[[x,y]_{\mathrm{trop}}\coloneq\{(\lambda\otimes x)\oplus(\mu\otimes y)\mid \lambda,\mu\in\R\}.\]

The above notion for convexity in \toms generalises this in a natural way:
In the realisable case \index{realisable!tropical oriented matroid} the convex hull $M_{AB}$ of two types contains all cells that intersect a tropical line segment between two points in open cells of types $A$ and $B$ in \emph{some} realisation of $M$. See Figure \ref{fig:conv_hull} for an illustration.

\begin{figure}[h]
\centering
\includegraphics[width=7cm]{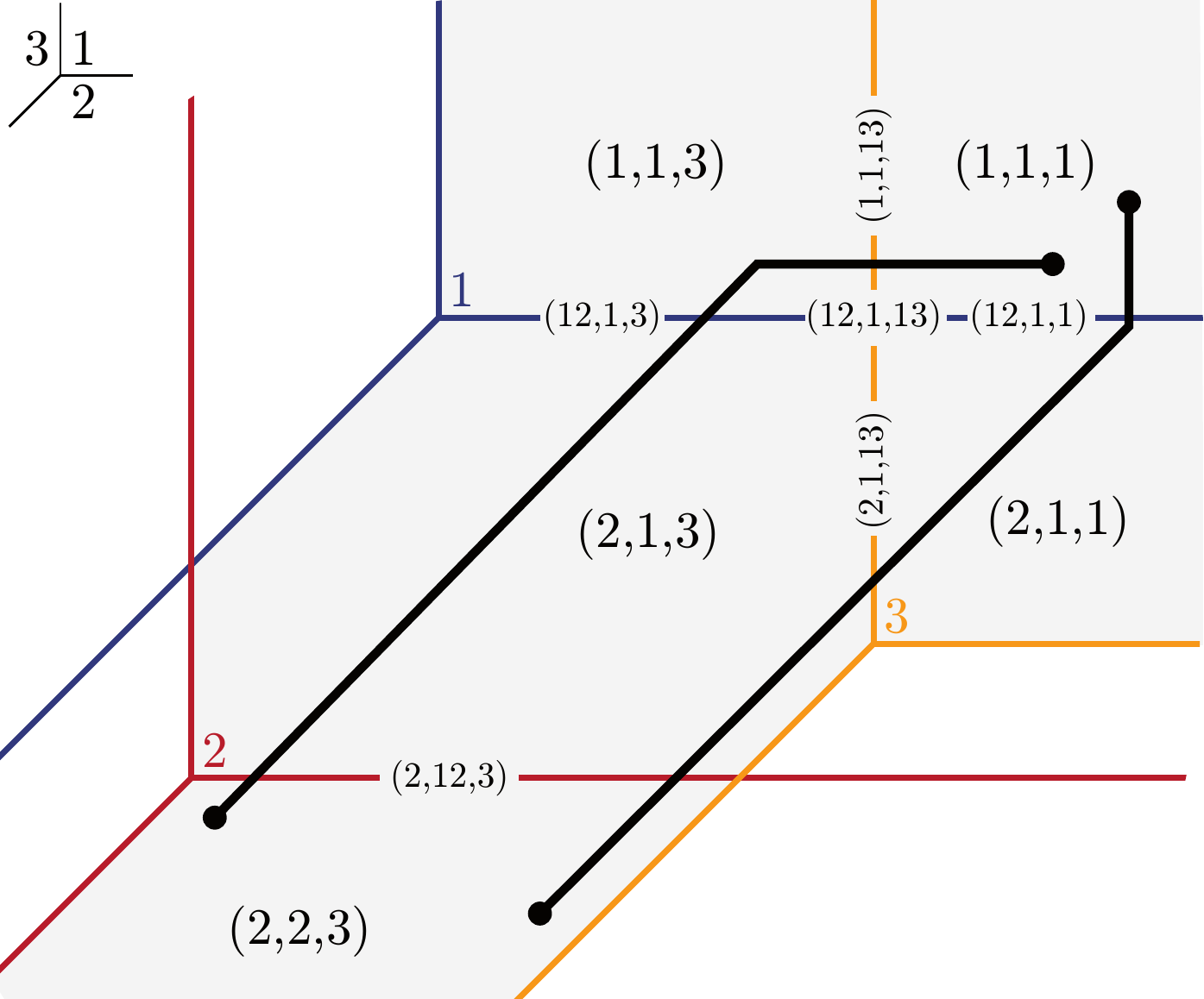}
\caption[The convex hull of two types in a realisable \tom.]{The convex hull of two types $A=(2,2,3)$, $B=(1,1,1)$ in a realisable \tom with parameters $(3,3)$. In this realisation every cell in the convex hull intersects a tropical line segment between points in $A$ and points in $B$. Note though that there are other realisations of the same \tom where this does not hold. (Imagine shifting the apex of the second tropical hyperplane further to the right until it is no longer possible to draw a line segment from $A$ to $B$ through the cell $(1,1,3)$.)}
\label{fig:conv_hull}
\end{figure}

\medskip
The following proposition establishes a connection between the combinatorial convex hull and the elimination property.
\begin{proposition}\label{prop:elim_iff_connected}
The types of the cells in a \mixsd $S$ of $\dilsimp n{d-1}$ satisfy the elimination property \tiff $ S_{AB}$ is path-connected (as a subcomplex of $S$) for every $A,B\in S$.
\end{proposition}

\begin{proof}
The convex hull $S_{AB}$ clearly contains each elimination of $A$ and $B$.
If $S_{AB}$ is path-connected then there is a path from $A$ to $B$ in $S_{AB}$. For any given $j\in[n]$ this path must contain a cell $C$ with $C_j=A_j\cup B_j$. Then $C$ works as elimination for $A$ and $B$ with respect to $j$.

Conversely, assume that $S$ satisfies the elimination property and fix $A,B\in S$. We have to show that there exists a path from $A$ to $B$ in $S_{AB}$. 

Denote $\dist(A,B)\coloneq \{ i\mid A_i\not\subseteq B_i, B_i\not\subseteq A_i \}$. If $\card{\dist(A,B)}=0$ then $A\cap B\in S_{AB}$ and we are done.
Otherwise choose some position $i\in\dist(A,B)$ and let $C$ denote the elimination of $A$ and $B$ with respect to $i$. Then $C\in S_{AB}$ and we will now show that \mbox{$\card{\dist(A,C)}, \card{\dist(B,C)}\leq \card{\dist(A,B)}-1$}.

Indeed consider $j\nin\dist(A,B)$. Then $j\nin\dist(A,C)$ follows immediately. Moreover, $i\in\dist(A,B)\setminus\dist(A,C)$. Thus $\card{\dist(A,C)}\leq\card{\dist(A,B)}-1$ and similarly for $\dist(B,C)$.

The claim then follows by iterating this process.
\end{proof}

\begin{corollary}\label{cor:conv_connected}
A convex set in a \tom is path-connected.
\end{corollary}
\begin{proof}
Since \toms satisfy the elimination property, Proposition \ref{prop:elim_iff_connected} implies that the convex hull of any two types is path-connected.
\end{proof}

\section{The Topological Representation Theorem}\label{sec:toprep}
This section comprises the long and winding road towards the Topological Representation Theorem for \toms. Note that a different version (with a different definition of \tphpas) is contained in \cite{toprep1}.

\medskip
We first introduce \tphps:
\begin{definition}[{\Cf \cite[Definition 4.3]{toprep1}}]\label{def:tphp}
A \defn{tropical pseudohyperplane} is the image of a \thp under a PL-homeo\-mor\-phism of $\TP^{d-1}$ that fixes the boundary.
\end{definition}

By \cite[Theorem 4.2]{toprep1} the Poincaré dual of a \dnmixsd is a family of \tphps.

\subsection{Linear and affine pseudohyperplanes}
Locally, (\ie in the parallelepiped cells of their \mixsds) we want \tphps to intersect as ``ordinary'' hyperplanes. We thus introduce arrangements of linear pseudohyperplanes on the basis of arrangements of pseudospheres as defined in Björner, Las Vergnas, Sturmfels, White and Ziegler \cite[Def.\ 5.1.3]{BLSWZ}.
\begin{definition}[\Cf {\cite[Definition 5.1.3]{BLSWZ}}]
\label{def:linear_phpa}
A \defn{pseudohyperplane} is a set that is PL-homeomorphic to a linear hyperplane.
A finite collection $\mc A=(H_e)_{e\in E}$ of pseudohyperplanes is called an \defn[arrangement!of (linear) phps]{arrangement of pseudohyperplanes} if the following conditions hold:
\begin{enumerate}
\item $H_A\coloneq \bigcap_{e\in A} H_e$ is a pseudohyperplane of smaller dimension for all $A\subseteq E$.
\item If $H_A\not\subseteq H_e$ for $A\subseteq E, e\in E$ and $H^+_e$ and $H^-_e$ are the two sides of $H_e$, then $H_A\cap H_e$ is a pseudohyperplane in $H_A$ with sides $H_A\cap H_e^+$ and $H_A\cap H_e^-$.
\item The intersection of an arbitrary collection of closed sides is a ball.
\end{enumerate}
\end{definition}

We now define arrangements of \emph{affine} pseudohyperplanes as a generalisation of the above:

\begin{definition}\label{def:affine_phpa}
An \defn[arrangement!of affine p]{arrangement of affine pseudohyperplanes} is a collection $\mc A$ of pseudohyperplanes such that for any $\mc A'\subseteq\mc A$
either $\bigcap_{a\in \mc A'} H_{a}=\emptyset$ or
$\mc A'$ is an arrangement of linear pseudohyperplanes as defined in Definition \ref{def:linear_phpa}.
\end{definition}

\begin{proposition}\label{prop:intersection_of_aff_phss_connected}
The intersection of any number of closed  pseudohalfspaces in an arrangement of affine pseudohyperplanes in $\R^d$ is path-connected.
\end{proposition}
\begin{proof}
Let $H_i, 1\leq i\leq n$ be affine \phps in $\R^d$ and denote by $H_i^+$ the corresponding closed pseudohalfspaces.
\medskip

The proof will be done by induction on the number $n$ of \phps, the case where $n=1$ being trivially clear.

Assume $n\geq2$ and choose two points $x,y$ in $\bigcap_{i=1}^n H_i^+$. By induction there is a path $p$ from $x$ to $y$ in $\bigcap_{i=1}^{n-1} H_i^+$. Assume \twlog that whenever $p$ intersects $H_n$, it crosses it. (Otherwise we can modify $p$ to achieve this.)

If $H_n$ does not intersect $p$, we are done since then $p\subseteq \bigcap_{i=1}^n  H_i^+$. If $H_n$ intersects $p$, then it does so an even number of times. (Walking along $p$, at each intersection point we switch between $H_n^+$ and $H_n^-$.) Let $q,q'$ be the first two intersection points. 

We have to find a path $p'$ from $q$ to $q'$ in $\bigcap_{i=1}^n H_i^+$. We will prove the existence of $p'$ by induction on the dimension $d$.

We start with the case $d=2$. \Ie the $H_i$ are $1$-dimensional.
Define $p'$ to be the segment of $H_n$ between $q$ and $q'$.
Then $p'$ lies in $\bigcap_{i=1}^{n-1} H_i^+$. Indeed, assume that there is $1\leq i\leq n-1$ such that $H_i\cap p'\neq\emptyset$. Then $p'$ and the segment of $p$ between $q$ and $q'$ form a PL-$1$-sphere $S$. Since the intersection of $H_i$ and $H_n$ is a crossing, $H_i$ enters the interior of $S$ and hence has to intersect $S$ a second time by the Jordan curve theorem. Since $H_i\cap p=\emptyset$, there is a second intersection point of $H_i$ with $p'$. 
This is a contradiction. 

See Figure \ref{fig:affine_hss} for an illustration.

\begin{figure}[h]
\centering
\includegraphics[height=5cm]{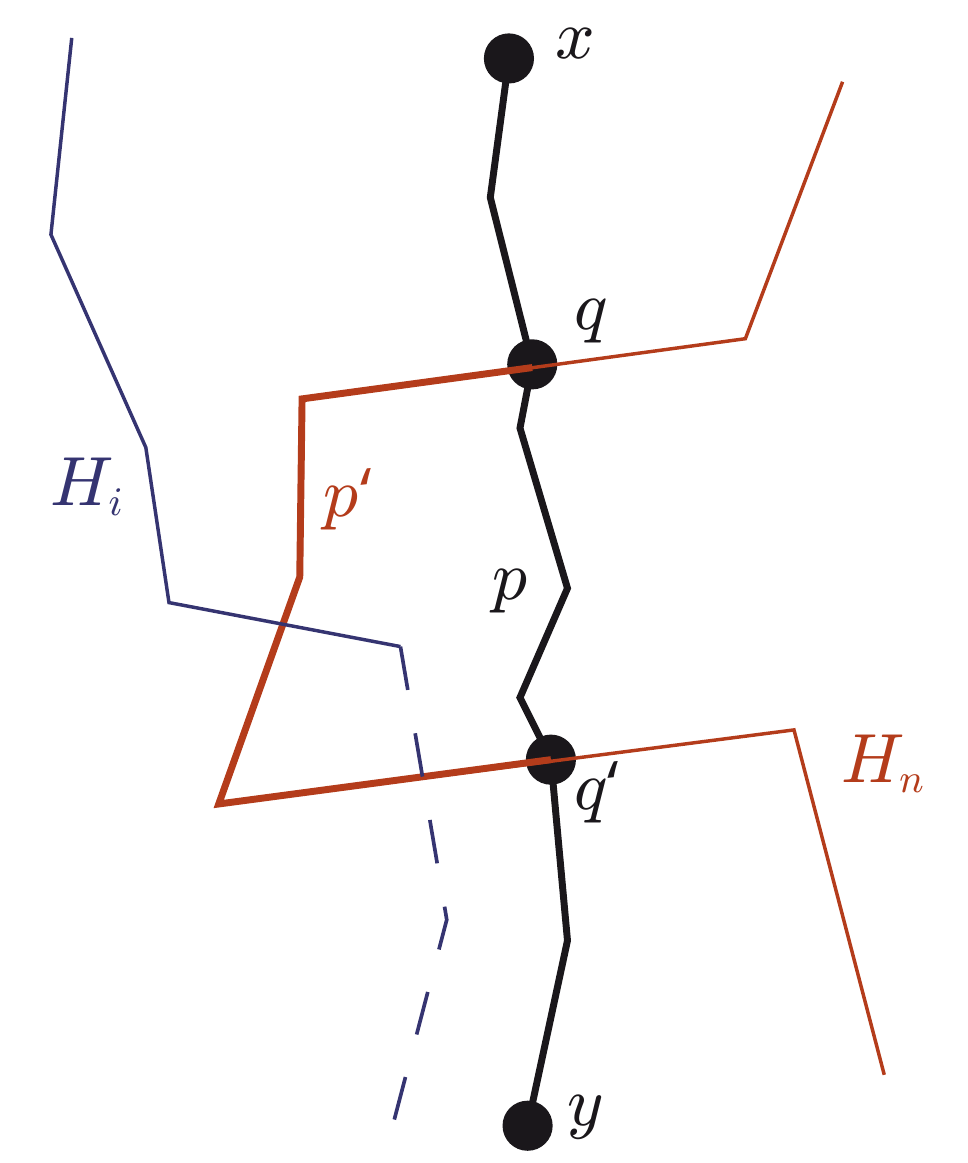}
\caption{The $2$-dimensional situation in the proof of Proposition \ref{prop:intersection_of_aff_phss_connected}.}
\label{fig:affine_hss}
\end{figure}
\medskip
Now assume $d\geq3$.

Denote $H'_i\coloneq H_i\cap H_n$ and $(H'_i)^+\coloneq H_i^+\cap H_n$ for $1\leq i\leq n-1$.  Then $\{H'_i\}$ is an arrangement of affine \phps in $H_n\PLhom\R^{d-1}$ and $q,q' \in \bigcap_{i=1}^{n-1}(H'_i)^+$. By induction this set is path-connected.

Hence there is a path $p'$ from $q$ to $q'$ in $\bigcap_{i=1}^{n-1}(H'_i)^+\subset\bigcap_{i=1}^n H_i^+$. Replace the segment of $p$ between $q$ and $q'$ by $p'$ and continue in the same way for the other intersection points.

\medskip
Thus, we constructed a path from $x$ to $y$ in $\bigcap_{i=1}^n H_i^+$. Since $x$ and $y$ were arbitrary, this proves that $\bigcap_{i=1}^n H_i^+$ is path-connected.
\end{proof}

\subsection{Arrangements of \tphps II}
We now define arrangements of \tphps. Note that a second definition of \tphpas is given in \cite[Definition 4.3]{toprep1}. We will eventually see that both definitions are equivalent.

Let $H$ be a $(d-2)$-dimensional \tphp in $\T^{d-1}$. Then $H$ divides $\T^{d-1}\setminus H$ into $d$ connected components $S_1,\ldots,S_d$, the \defn[sectors!of a \tphp]{open sectors} of $H$. The closure of any union $\bigcup_{i\in I}S_i$ with $\emptyset\neq I\subset [d]$ will be called a \defn[pseudohalfspace!of a \tphp]{(tropical) pseudohalfspace} of $H$. We denote by \[H_I\coloneq\bnd\,{\cl{\bigcup_{i\in I}S_i}}=\bnd\,{\cl{\bigcup_{i\notin I}S_i}}\] the boundary of the pseudohalfspace and by \[H_I^+\coloneq\cl{\bigcup_{i\in I}S_i}\setminus H_I,\quad\text{respectively}\quad H_I^-\coloneq\cl{\bigcup_{i\notin I}S_i}\setminus H_I\] the two open pseudohalfspaces. Note that the boundary $H_I$ of a tropical pseudohalfspace is a (linear) pseudohyperplane with sides $H_I^+$ and $H_I^-$.

An $(n,d)$-halfspace system is a tuple $\mc I=(I_1,\ldots,I_n)$ with $\emptyset\neq I_i\subset [d]$ for each $1\leq i\leq n$. Given a halfspace system $\mc I$ and a collection $\mc A=(H_i)_{i\in[n]}$ of $n$ \tphps we write 
\[\mc A_{\mc I}\coloneq\{H_{i,I_i}\mid 1\leq i\leq n\}.\]

The following definition of \tphpas is motivated by Propositions \ref{prop:elim_iff_connected} and \ref{prop:intersection_of_aff_phss_connected}, \ie by the fact that we want to show that the combinatorial convex of hull of two types is path-connected and know that the intersection of affine pseudohalfspaces is so.
\begin{definition}
\label{def:tphpa2}
An \defn[arrangement!of \tphps]{arrangement of tropical pseudohyperplanes} (in weakly general position) is a collection $\mc A$ of $n$ \tphps in $\T^{d-1}$ such that  $\mc A_{\mc I}$ forms an arrangement of affine pseudohyperplanes as defined in Definition \ref{def:affine_phpa} for every $(n,d)$-halfspace system $\mc I$.
\end{definition}

See Figure \ref{fig:tphpas} for examples of arrangements of \tphps in $\T^3$.
\medskip

\begin{figure}
\centering
\includegraphics[height=3.5cm]{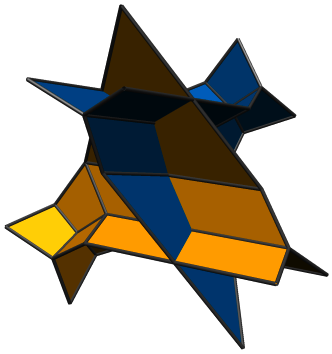}\hspace{1cm}\includegraphics[height=3.5cm]{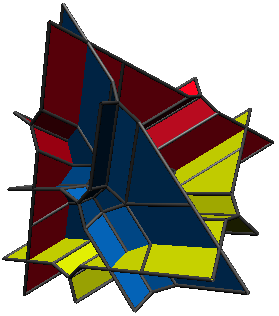}\hspace{1cm}\includegraphics[height=3.5cm]{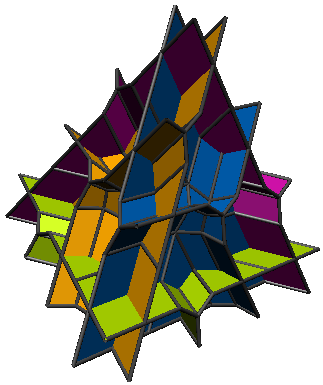}
\caption{Arrangements of $2$-dimensional \tphps that are dual to \mixsds of dilated simplices. The arrangement on the right is non-realisable. The pictures were produced with the \texttt{polymake} extension \texttt{tropmat} \cite{tropmat}.}
\label{fig:tphpas}
\end{figure}

For a set $I\subseteq[n]$ we denote its complement by $\compl I\coloneq[n]\setminus I$.
For a \tphp $H$ and a halfspace $\emptyset\neq I\subset[d]$ we define \[\begin{split}\mc T_I:\mc C(H)&\to\{+,-,0\}\\C&\mapsto\begin{cases}
+ &\tif C\subseteq I,\\ 
-&\tif C\subseteq \compl {I}=[d]\setminus I,\\
0&\other.
\end{cases}\end{split}
\] 
Now let $\mc A$ be a \tphpa and $\mc C(\mc A)$ the induced cell decomposition of $\T^{d-1}$. For $\mc A'\subseteq \mc A$
we define
\[\begin{split}\mc T_{\mc I}:\mc C(\mc A')&\to\{+,-,0\}^{\mc A'}:\\
C&\mapsto (\mc T_{I_i}(C_i))_i\\
\end{split}\] and 
\[\mc L(\mc A',\mc I)\coloneq \{\mc T_{\mc I}(C)\mid C\in\mc C(\mc A')\}.\] 

\begin{proposition}\label{prop:halfspace_system_om}
Let $M$ be a \tom in general position and $S$ its corresponding fine \dnmixsd. Moreover, fix a halfspace system $\mc I$. Then 
\begin{itemize}
\item either $0\nin \mc L(\mc A',\mc I)$ or
\item $(\mc L(\mc A',\mc I),\mc A')$ is an oriented matroid with covectors $\{0,+,-\}^{\card{\mc A'}}$.
\end{itemize}
\end{proposition}

\begin{proof}
Let $\mc L\coloneq \mc L(\mc A',\mc I)$ and assume $0\in\mc L$. (Otherwise there is nothing to prove.)

We  show that $\mc L=\{+,-,0\}^{\mc A'}$.
Choose $A\in \inv{\mc T_{\mc I}}(0)$. Then one can for any $X\in\{+,-,0\}^{\mc A'}$ construct a type $B\subseteq A$ with $\mc T_{\mc I}(B)=X$.
So define $B$ by \[B_i=\begin{cases}A_i&\tif X_i=0\\ A_i\cap \mc I_i&\tif X_i=+\\ A_i\cap \mc \compl{\mc I_i}&\tif X_i=-.\end{cases}\] 
Then $B\subseteq A$ and since $M$ is in general position $B$ is a refinement of $A$. Moreover, $\mc T_{\mc I}(B)=X$.
\end{proof}

\medskip
If $J_i\subseteq [d]$ for each $i\in[n]$ and the $J_i$ are pairwise disjoint then we denote by $J_1\disjoint\ldots\disjoint J_n$ the \emph{partition} of $\bigcup_i J_i$ into the $J_i$.

\medskip
Now let $\mc J=(J_1,\ldots,J_n)$ be an $n$-tuple of partitions of $[d]$. \Ie  $J_i=(J_{i,1}\disjoint\ldots\disjoint J_{i,k_i})$ is a partition of  $[d]$ for each $i\in[n]$.
For a \tom $M$ denote by \[M_{\mc J}\coloneq\{A\in M\mid A_i\cap J_{i,k}\neq\emptyset, i\in[n],k\in[k_i]\}\] 
the set containing all types in $M$ all of whose entries intersect each element in the according partition.
As before, let $\mc I=(I_1,\ldots, I_n)$ be an $n$-tuple of non-empty subsets of $[d]$. Then we denote \[M_{\mc I}\coloneq\{A\in M\mid A_i\subseteq I_i, i\in[n]\}.\]
Finally, we define \[M(\mc I,\mc J)\coloneq M_{\mc I}\cap M_{\mc J}.\]
See Figure \ref{fig:constructible} for an illustration of $M(\mc I,\mc J)$.

\newcommand{\mij}{M(\mc I,\mc J)}

\begin{lemma}\label{lem:mij_connected_pure}
Let $M$ be a \tom in general position.
Then $M(\mc I,\mc J)$,
if non-empty, is connected and pure of dimension $d+n-1-\sum\card{J_i}$.
\end{lemma}
\begin{proof}
We first show that $\mij$ is connected: Let $A,B\in \mij$. Then $A_i, B_i\subseteq I_i$ and $A_i\cap J_{i,k}, B_i\cap J_{i,k}\neq \emptyset$ for each $i\in[n]$ and $k\in[k_i]$. But this implies $A_i\cup B_i\subseteq I_i$ and $(A_i\cup B_i)\cap J_{i,k}\neq\emptyset$.
Hence $\mij$ is convex in the sense of Definition \ref{def:tom_convex} and thus connected by Proposition \ref{prop:elim_iff_connected}.

It remains to show that $\mij$ is pure of the correct dimension. 
Let $A\in \mij$.  Since $A_i\cap J_{i,k}\neq\emptyset$, it follows that $\card{A_i}\geq \card{J_i}$ for each $i$. Hence $\dim A\leq d+n-1-\sum\card{J_i}$.

Since $M$ is in general position we can construct a type $B\subseteq A$ with $\card{B_i\cap J_{i,k}}=1$ for every $i,k$ by deleting sufficiently many elements from the entries of $A$. Then $\dim B=d-1-\sum(\card {J_i}-1)=d+n-1-\sum\card{J_i}$. Since $A$ was arbitrary this shows that any type in $\mij$ is contained in one of  dimension $d+n-1-\sum\card{J_i}$.
\end{proof}

For a cell complex $\mc C$ we denote by $\cl{\mc C}$ its \defn[closure!of a cell complex]{closure}, \ie $\cl{\mc C}$ consists of all cells of $\mc C$ and their faces.

\begin{lemma}\label{lem:mij_manifold}
Let $M, \mc I, \mc J$ as before. Then
$\cl\mij$ is a PL-manifold with boundary.
\end{lemma}
\begin{proof}
Denote $\mc M\coloneq \cl\mij$ and $\mc M'\coloneq M_{\mc J}$. Choose a cell $T\in\mc M$.
We first investigate the link $\lk_{\mc M'}T$.
The cells in $\lk_{\mc M'} T$ correspond to the cells in the star $\st_{\mc M'} T=\{C\in\mc M'\mid C\subseteq T\}$ and hence to 
certain refinements of $T$. 
First assume that $n=1=k_1$, \ie $\mc J=(J_1=(J_{11}))$. Then the cells in $\st_{\mc M'}T$ are in bijection with the proper subsets of $J_{11}\cap T_1$ ordered by reverse inclusion. Hence $\lk_{\mc M'}T$ is the boundary of a simplex of dimension $\card {(J_{11}\cap T_1)}-1$ (whose facets are labelled by $J_{11}\cap T_1$).

Since $M$ is in general position we can consider the $J_{ik}$ (for $i\in[n],k\in[k_i]$) independently. \Ie in general, $\lk_{\mc M}T$ is the boundary of a product of simplices (one for each $J_{ik}$) and hence a PL-sphere.\index{PL!sphere} Denote this sphere by $\mc S(T)$. See Figures \ref{fig:constructible_b} and \subref{fig:constructible_c} for an example.

\medskip
If in each position $i$ there is some $J_{ik}$ with $J_{ik}\cap T_i\subseteq I_i$ then $T$ is contained in the interior of $\mc M$ and $\lk_{\mc M}T=\mc S(T)$. Otherwise denote by $\mc B(T)$ the set of all faces of $\mc S(T)$ that do \emph{not} belong to $\lk_{\mc M}T$. Then define $\mc J'$ by replacing each $J_{i}$ in $\mc J$ by $(I_i\disjoint (J_{i1}\cap \compl {I_i})\disjoint\ldots\disjoint (J_{ik_i}\cap\compl {I_i}))$. 
Then $\cl{\mc B(T)}\cap\lk_{\mc M}T=M_{\mc J'}$ is a PL-sphere in $\mc S(T)$ with sides $\mc B(T)$ and $\lk_{\mc M}T$. By \cite[Lemma 5.1.1]{BLSWZ} this implies that $\lk_{\mc M}T$ is a PL-ball.

It remains to show that $\mc M$ has a boundary. If there is a cell $T$ whose link is a ball we are done.
Otherwise -- unless $\mc M$ consists of a single point -- we can always construct a cell in $\mc M$ whose dual (in the \mixsd corresponding to $M$) is contained in the boundary of $\dilsimp n{d-1}$. (Note that we have to view $\mc M$ as a manifold in $\TP^{d-1}$ for it to be compact.) Indeed the cells in the boundary of $\dilsimp n{d-1}$ are characterised by the fact that their types are unbounded, \ie there is some $i\in[n]$ not contained in any position of the type. The only situation, however, when $i\in[n]$ is contained in any cell in $\mc M$ is when any $J_i$ contains a singleton $\{i\}$. If this holds for every $i$ then $\mc M$ consists of one point only.
\end{proof}

\subsection{Constructibility}
In the proof of the Topological Representation Theorem for classical oriented matroids given in \cite{BLSWZ}, the \defn{shellability} of certain complexes plays a crucial role. In particular, the fact that a shellable PL-manifold is either a ball or a sphere is used in order to show that the subcomplexes which one would like to be pseudospheres actually are pseudospheres.

In the proof of the tropical analogue, we are going to apply a related but weaker notion, namely that of constructibility.

The notion of constructibility of a polytopal complex goes back to Hochster \cite{Hochster}.

\begin{definition}\label{def:constructible}
A polyhedral $d$-complex $C$ is \defn{constructible} if 
\begin{itemize}
\item 
$C$ consists of only one cell or
\item
$C=C_1\cup C_2$, where $C_1, C_2$ are $d$-dimensional constructible complexes and $C_1\cap C_2$ is a $(d-1)$-dimensional constructible complex.
\end{itemize}
\end{definition}

\begin{proposition}\label{prop:mij_constructible}
Let $M$, $\mc I,\mc J$ as before. 
Then $M(\mc I,\mc J)$ is constructible.
\end{proposition}
\begin{proof}
We are done if $\mij$ consists of one (maximal) cell only. Otherwise there are two maximal cells $A$ and $B$. By  Lemma \ref{lem:mij_connected_pure} above (and the fact that $A,B$ are maximal) we then have $\card{A_i}=\card{B_i}$ and $\card{A_i\cap J_{i,j}}=\card{B_i\cap J_{i,j}}=1$ for every $i$ and $j$. 

There is some position $k$ where $A$ and $B$ differ. Moreover,
there is some $\ell$ with $J_{k,\ell}\cap A_k\neq J_{k,\ell}\cap B_k$. Let $a\in J_{k,\ell}\cap A_k, b\in J_{k,\ell}\cap B_k$. (Note that $a$ and $b$ are unique.)

Now form $\mc J_0$ by splitting $J_{k,\ell}$ so that $a$ and $b$ are in different sets. Moreover, form $\mc I_1,\mc I_2$ by 
removing $a$, respectively $b$ from $I_k$.
Then $\cl\mij=\cl{M(\mc I_1,\mc J)}\cup \cl{M(\mc I_2,\mc J)}$ 
and $\cl{M(\mc I_1,\mc J)}\cap \cl{M(\mc I_2,\mc J)}=\cl{M(\mc I,\mc J_0)}$. Moreover, $A\in \cl{M(\mc I_1,\mc J)}, B\in \cl{M(\mc I_2,\mc J)}$.
By the above lemma, $M(\mc I_1,\mc J), M(\mc I_2,\mc J), M(\mc I,\mc J_0)$ are connected and pure and of the right dimensions. 
By induction these three sets are constructible and hence $\cl\mij$ is constructible.
\end{proof}

See Figure \ref{fig:constructible} for an illustration.
\bigskip
\begin{figure}[t]
\centering
\subfigure[A $2$-dimensional \tphp. The $2$-faces are labelled by their types.]{\includegraphics[width=4.3cm]{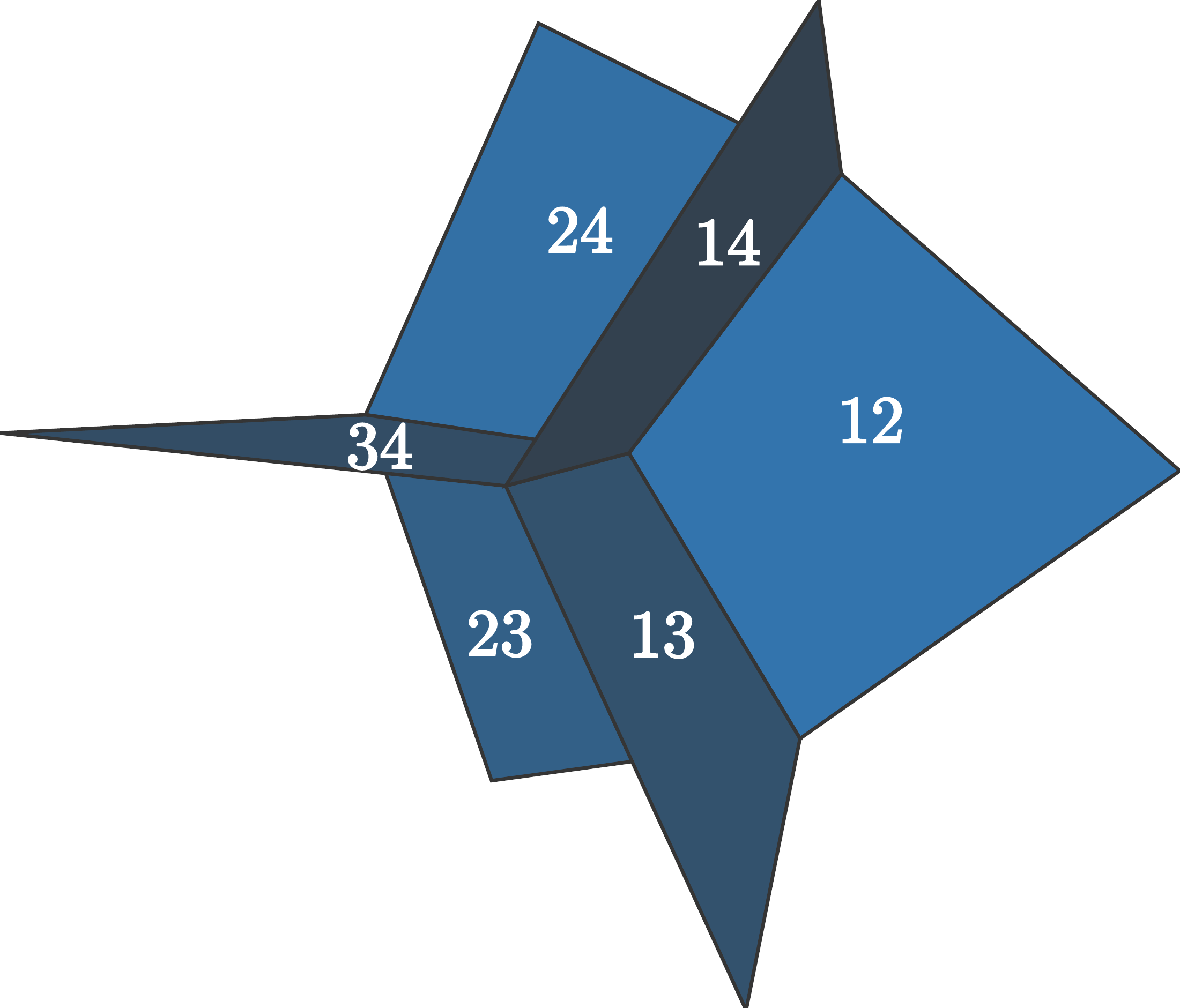}\label{fig:constructible_a}}
\qquad
\subfigure[The subcomplex $\mij$ for $\mc I={[4]},\ \mc J=(14 \disjoint 23)$ -- a $2$-dimensional PL-ball. The link of $1234$ is drawn in light grey.]{\includegraphics[width=4.3cm]{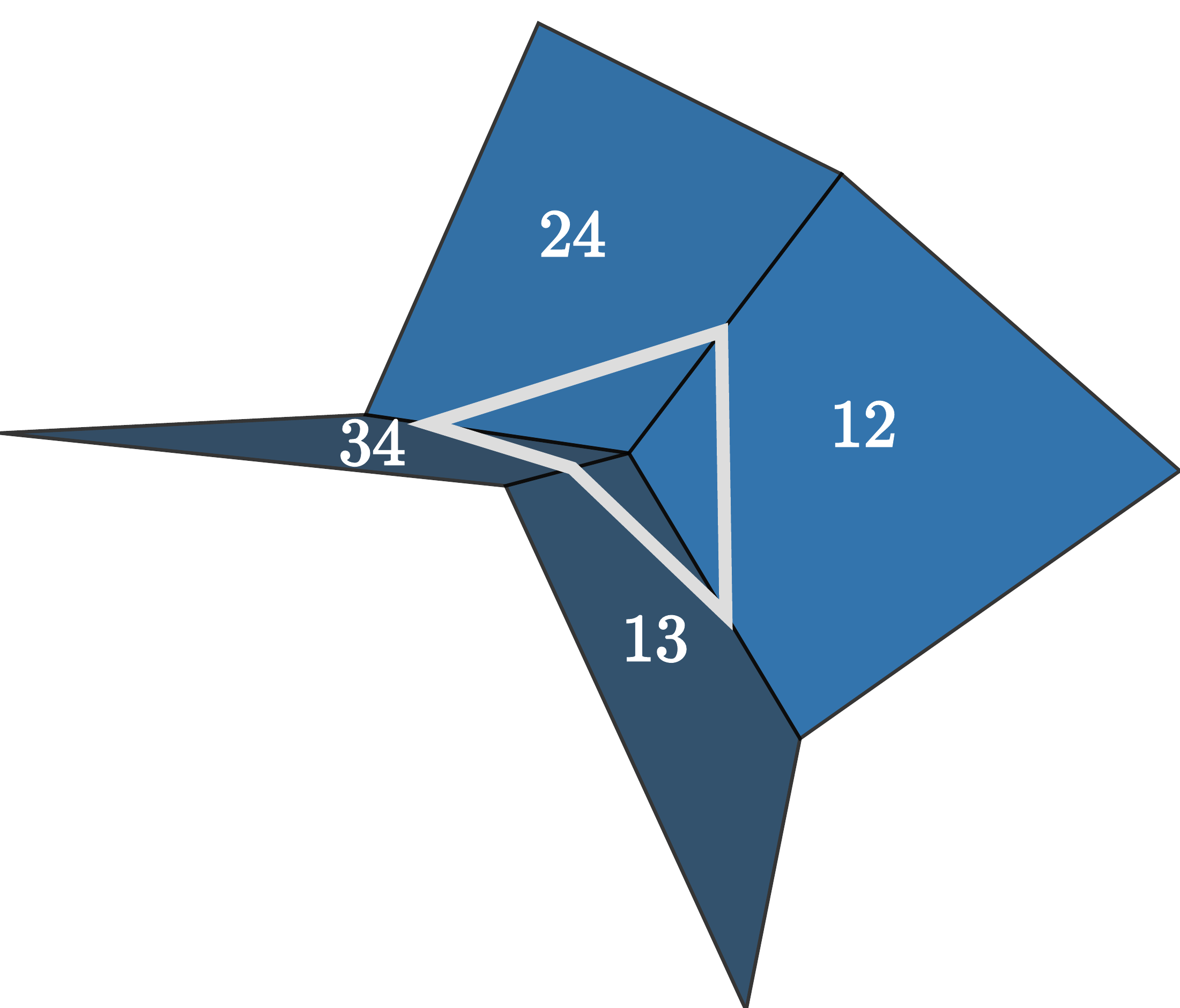}\label{fig:constructible_b}}
\qquad\subfigure[The subcomplex $M(\mc I,\mc J_0)$ for $\mc I={[4]},\ \mc J_0=(1\disjoint 4 \disjoint 23)$ -- a $1$-dimensional PL-ball -- and its sides $M(\mc I_1, \mc J)$ and $M(\mc I_2,\mc J)$.]{\label{fig:constructible_c}\includegraphics[width=4.3cm]{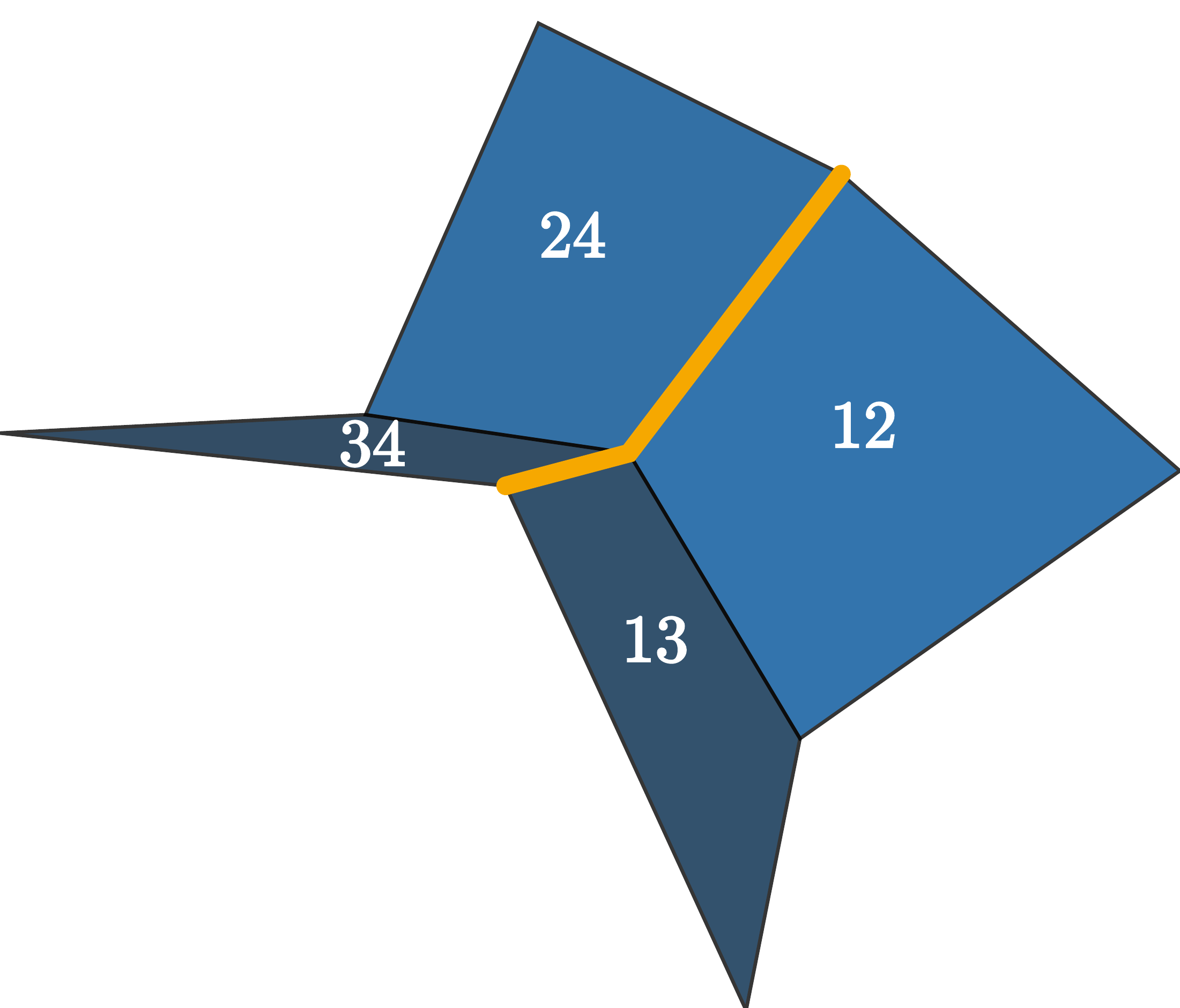}}
\caption{Assume in the proof of Proposition \ref{prop:mij_constructible} we have $n=1,d=4$, \ie we are dealing with a $2$-dimensional \tphp as depicted in Figure \subref{fig:constructible_a}. Moreover, assume we have $\mij$ with $\mc I=[4], \mc J=(14\disjoint 23)$. The complex $\mij$ is depicted in Figure \subref{fig:constructible_b}.\newline 
Now let $A=13,B=24$. As in the proof we see that $\card{A_1}=\card{B_1}$ and $\card{A_1\cap J_{1i}}=\card{B_1\cap J_{1j}}=1$ for every $i$ and $j$. We have $k=1$ and we may choose $\ell=1$. Then we get $a=1,b=4$ as the unique elements in $A_1\cap J_{11},B_1\cap J_{11}$. We form $\mc J_0=(1\disjoint 4\disjoint 23)$ by splitting $J_{k\ell}=14$. Moreover, we set $\mc I_1=234$ and $\mc I_2=123$. This situation is depicted in Figure \subref{fig:constructible_c}.
}
\label{fig:constructible}
\end{figure}

The above lemmas together with a theorem by Zeeman (\cite{Zeeman}, ``A constructible manifold with a boundary is a ball.'') yield:
\begin{proposition}\label{prop:mij_ball}
Let $M$ be a \tom in general position.
Then $\mij$ is a PL-ball.
\end{proposition}
\begin{proof}
$\cl\mij$ is constructible and pure of dimension $d+n-1-\sum\card{J_i}$ by Lemma \ref{lem:mij_connected_pure} and Proposition \ref{prop:mij_constructible}.

By Lemma \ref{lem:mij_manifold}, $\cl\mij$ is a PL-manifold with boundary and hence a PL-ball by Zeeman's theorem.
\end{proof}

\begin{corollary}\label{cor:tom_types_balls}
Let $M$ be a \tom in general position and $S$ its corresponding fine \dnmixsd. Moreover, choose a halfspace system $\mc I$ and $X\in\{+,-,0\}^n$. Then $\inv{\mc T_{\mc I}}(X)$ is a PL-ball of dimension $d-1-\card {z(X)}$, where $z(X)$ denotes the zero set of $X$.
\end{corollary}
\begin{proof}
Define $\mc I'=(I_1',\ldots,I_n')$ by \[I_i'\coloneq\begin{cases} I_i&\tif X_i=+,\\ \compl{I_i}&\tif X_i=-,\\ [d] &\tif X_i=0\end{cases}\] and $\mc J=(J_1,\ldots,J_n)$ by \[J_i\coloneq \begin{cases}[d] &\tif X_i\in\{+,-\},\\ I_i\disjoint \compl {I_i}&\tif X_i=0.\end{cases}\]

Then $\inv{\mc T_{\mc I}}(X)=M(\mc I',\mc J)$ and hence the claim follows from Proposition \ref{prop:mij_ball}.
\end{proof}

We are now ready to prove the following version of the Topological Representation Theorem for \toms:

\begin{theorem}\label{thm:toprep2}
Every \tom in general position can be realised by an \atphp as in Definition \ref{def:tphpa2}.
\end{theorem}
\begin{proof}
Let $M$ be a \tom in general position, $S$ the fine \dnmixsd corresponding to $M$ and $\mc A$ the family of \tphps induced by $S$. We have to show that $\mc A'_{\mc I}$ is an arrangement of affine pseudohyperplanes for each $\mc A'\subseteq\mc A$ and halfspace system $\mc I=(I_1,\ldots,I_n)$.

So assume that $\bigcap \mc A'_{\mc I}\neq\emptyset$, \ie $0\in\mc L(\mc A',\mc I)$. Hence by Proposition \ref{prop:halfspace_system_om} $(\mc L(\mc A',\mc I),\mc A')$ is an oriented matroid given by its covectors.

We have to show that $\mc A'_{\mc I}$ satisfies the axioms in Definition \ref{def:linear_phpa}.
\begin{enumerate}
\item Let $A\subseteq \mc A'_{\mc I}$. We have to show that $H_A\coloneq \bigcap_{a\in A} H_a$ is a PL-ball.\index{PL!ball} So let $\mc I'=(I_1',\ldots,I_n')$ with $I_i'=[d]$ for each $i$ and $\mc J=(J_1,\ldots,J_n)$ with \[J_i=\begin{cases}I_i\disjoint \compl{I_i}& \tif i\in A,\\ [d] & \other. \end{cases}\]
Then $H_A=M(\mc I',\mc J)$.

\item Assume $e\nin A$. Then $H_A\not\subseteq H_e$. We have to show that $H_A\cap H_e$ is a pseudohyperplane in $H_A$ with sides $H_A\cap H_e^+$ and $H_A\cap H_e^-$. 

To this end let $\mc I',\mc J$ as before. Moreover, define $\mc I'_1,\mc I'_2$ by
\[I'_{1,i}=\begin{cases}I_i&\tif i=e,\\ [d]&\other, \end{cases}\] \[I'_{2,i}=\begin{cases}\compl{I_i}&\tif i=e,\\ [d]&\other \end{cases}\] and $\mc J_0$ by 
\[J_{0,i}=\begin{cases}I_i\disjoint \compl{I_e}&\tif i=e,\\ J_i&\other. \end{cases}\] 
Then $H_A\cap H_e=M(\mc I',\mc J_0)$, $H_A\cap H_e^+=M(\mc I_1,\mc J)$ and $H_A\cap H_e^-=M(\mc I_2,\mc J)$. Since $\bigcap \mc A'_{\mc I}\neq\emptyset$, each of  $H_A\cap H_e, H_A\cap H_e^+$ and $H_A\cap H_e^-$ is non-empty by Proposition \ref{prop:halfspace_system_om}.

Hence $H_A\cap H_e, H_A\cap H_e^+$ and $H_A\cap H_e^-$ are PL-balls of the correct dimensions. Moreover, $\cl{H_A\cap H_e^+}\cap \cl{H_A\cap H_e^-}=H_A\cap H_e$ and hence $H_A\cap H_e^+$ and $H_A\cap H_e^-$ are the sides of $H_A\cap H_e$.

\item We have to show that the intersection of an arbitrary collection of closed sides is a PL-ball. This follows directly from Corollary \ref{cor:tom_types_balls}.
\qedhere
\end{enumerate}
\end{proof}

\section{The elimination property}\label{sec:elim}
This section is about the all important elimination property. 
Recall that by Oh and Yoo \cite[Proposition 4.12]{Oh/Yoo} the elimination property holds for fine \dnmixsds.
In this section we apply the Topological Representation Theorem \ref{thm:toprep2} to extend this to all \dnmixsds.

\subsection{Blowing up hyperplanes in a \mixsd}

Let $S$ be a fine \dnmixsd and fix $i\in[n]$. The following construction is an inverse of the deletion operation and yields a \mixsd of $\dilsimp N{d-1}$ ($N>n$) by ``blowing up'' one \tphp in the dual arrangement.  

We have to fix some notation:
Let $S$ be a fine \dnmixsd. For $\emptyset\neq I\subset[n]$ we denote by $\restr SI$ the \mixsd of $\dilsimp{n}{\card I-1}$ induced by $S$ on the $I$-face of $\dilsimp n{d-1}$. \Ie $\restr SI$ is the contraction $\contraction S{\compl I}$ of $S$ with the complement of $I$.

\begin{definition} Let $S, S'$ be fine \dnmixsds, respectively $\dilsimp{n'}{d-1}$.
Let $C\in S$ be a cell. Then the \defn[blow-up!of a \mixsd]{blow-up} of $C$ with respect to $S'$ at position $i$ is the set of $(n+n'-1,d)$-types \[C\vee_i S'\coloneq\{(\deletion Ci,X)\mid X\in\restr {S'}{C_i}\}.\]
That is, we subdivide the $C_i$-face of $C$ as $\restr{S'}{C_i}$.

Moreover, the \defn[blow-up!of a \mixsd]{blow-up} of $S$ with respect to $S'$ at position $i$ is \[S\vee_i S'\coloneq\bigcup_{C\in S}C\vee_i S'.\]
\end{definition}

See Figure \ref{fig:blow_up} for an example.

The following lemma follows easily:
\begin{lemma}
The types in the blow-up $S\vee_i S'$ yield a fine \mixsd of $\dilsimp{N}{d-1}$ with $N\coloneq n+n'-1$.
\end{lemma}
\begin{proof}
It is clear that each type corresponds to a Minkowski cell inside $\dilsimp N{d-1}$ and that the cells cover $\dilsimp N{d-1}$. It remains to show the intersection property.

Let $A=A_S\vee_i A_{S'}, B=B_S\vee_iB_{S'}$ be two cells in $S\vee_iS'$. We have to show that $A$ and $B$ are comparable. Since $S$ is a \mixsd, $A_S$ and $B_S$ are comparable, \ie $\compgr{A_S}{B_S}$ is acyclic. The same holds for $\compgr{A_{S'}}{B_{S'}}$.

Now consider the comparability graph $\compgr AB$. This has the same vertex set $[d]$ and all edges from $\compgr{A_S}{B_S}$ accounting for positions different from $i$ and all edges from $\compgr{A_{S'}}{B_{S'}}$.

For position $i$, the graph $\compgr{A_S}{B_S}$ contains one edge (directed or undirected) between $a$ and $b$ for every $a\in A_{S,i}, b\in B_{S,i}, a\neq b$. The edge set of $\compgr{A_{S'}}{B_{S'}}$ is a subset of the set of these edges. An undirected edge in $\compgr{A_S}{B_S}$ might, however, correspond to a directed one in $\compgr{A_{S'}}{B_{S'}}$. 
Since $S'$ is a \mixsd, the graph $\compgr{A_{S'}}{B_{S'}}$ is acyclic.

Hence it remains to exclude that an undirected cycle in $\compgr{A_S}{B_S}$ becomes a directed one in $\compgr AB$. But since $S$ is fine, for any undirected edge in $\compgr{A_S}{B_S}$ there is a unique position accounting for this edge. Moreover, any undirected cycle in $\compgr{A_S}{B_S}$ would yield a cycle in the type graphs of $A_S$ and $B_S$ which do not exist since $S$ is fine.
\end{proof}

Now fix some  permutation $\pi$ of $[d]$.
Let $S_\pi$ be the $n$-placing extension\index{n-placing@$n$-placing} of $\nsimplex{d-1}$ with respect to $\pi$.
Then we define the \defn[blow-up!of a \mixsd]{blow-up} of the $i$-th \tphp in $S$ with respect to $\pi$ by \[S_{i,\pi}\coloneq S\vee_i S_\pi.\]

In the dual setting of an \atphp this blow-up operation corresponds to adding a slightly shifted copy of the $i$-th tropical pseudohyperplane.

\bigskip

It is more difficult to define the blow-up of a \tphp in a \dnmixsd which is not fine.
Let $S$ be a \dnmixsd, $i\in[n]$ and $\pi=(\pi_1,\ldots,\pi_d)\in\Sym_d$. We also denote by $\overline\pi\coloneq (\pi_d,\ldots,\pi_1)$ the permutation obtained by reversing $\pi$.

Then the blow-up of the $i$-th \tphp has the following full-di\-men\-sion\-al cells:
\begin{itemize}
\item If $A=(A_1,\ldots,A_n)$ is a full-dimensional cell in $S$ with $\card{A_i}=1$ (\ie $A$ is not contained in the $i$-th hyperplane), then $(A,A_i)$ is a maximal cell in $S_{i,\pi}$.
\item If $A=(A_1,\ldots,A_n)$ is a full-dimensional cell in $S$ with $\card{A_i}\geq 2$ then $(A,\{\pi_d\})$ is a maximal cell in $S_{i,\pi}$.

\item Finally, the maximal cells corresponding to the new hyperplane are constructed as follows: Let again $S_\pi$ denote the $n$-placing extension of $\dilsimp n{d-1}$ with respect to $\pi$. Let $P$ be an ordered partition of $[d]$ that has $\overline\pi$ as a refinement. (\Ie neighbouring entries of $\overline\pi$ may be combined into one set.) Moreover, let $A$ be a full-dimensional cell in $S$ with $\card{A_i}\geq 2$. Define $B\coloneq \restr AP$ and let $C$ be the unique maximal cell in $S_\pi$ with $C_1=B_i$. Then $(B,C_2)$ is a maximal cell in $S_{i,\pi}$.
\end{itemize}

\subsection{Approximation by blow-ups}
\label{sec:phpa_to_tom}

\begin{figure}[b]
\centering
\includegraphics[width=2.8cm]{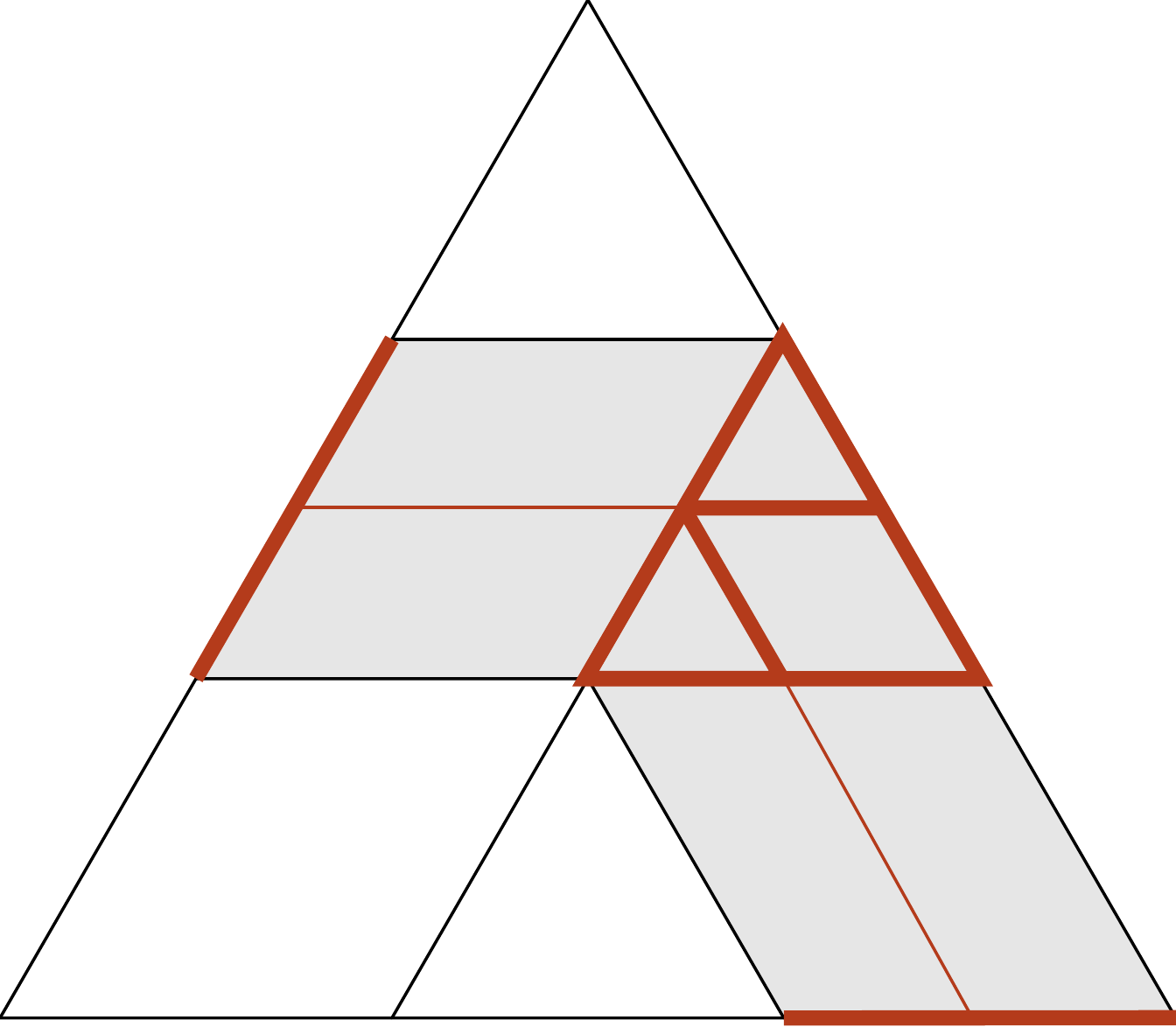}\qquad\qquad \includegraphics[width=2.8cm]{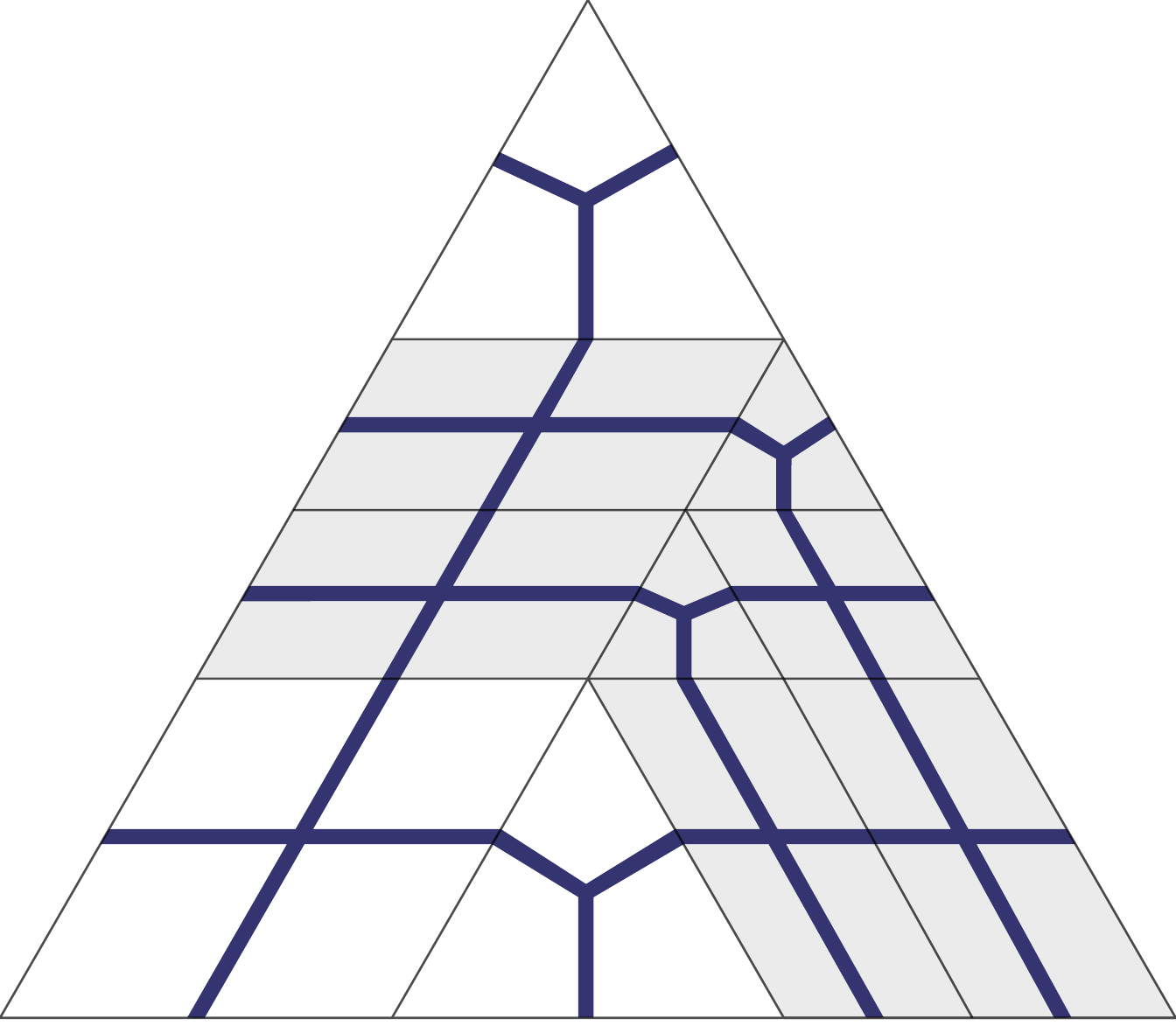} 
\vspace{-2mm}
\caption{The blow-up of a \mixsd of $\dilsimp32$ with respect to one of $\dilsimp22$.  The cells in the shaded hyperplane are subdivided according to the subdivision of the small simplex. The according \tphpa is drawn on the left.}
\label{fig:blow_up}
\end{figure}

In this section we prove that \tphpas as defined in Definition \ref{def:tphpa2} satisfy the elimination property and use this to show the same for all \dnmixsds.

\smallskip
Since it simplifies the presentation we assume all arrangements of \tphps in this section to come from a (fine) \dnmixsd. \Ie we only consider \tphpas which are dual to a fine \dnmixsd.

\smallskip
Let $H$ be a \thp with apex $0$. Recall that $H_I$ denotes the boundary of the tropical halfspace separating the points with types in $I$ from those with types in the complement $\compl I$.
For $p\in \T^{d-1}$ and $\emptyset\neq I\subseteq[d]$ denote  $H_{I,p}\coloneq H_I-p$, \ie we shift the apex of $H_I$ to  $p$. 
For $\emptyset\neq I\subseteq[d]$ denote by $T_I$ the set of all points of type $I$.
Let \[\mc F\coloneq \{\aff T_I\mid I\in\tbinom {[d]}2 \},\] \ie $\mc F$ is an arrangement of linear hyperplanes in $\T^{d-1}$. In fact, $\mc F$ is the arrangement of reflection hyperplanes corresponding to the Coxeter group\index{Coxeter!group} $A_d$.
The connected components (\defn[]{sectors}) of $\T^{d-1}\setminus (\bigcup\mc F)$ correspond one-to-one to the permutations of $[d]$: Again, view $\mc F$ embedded in the simplex $\nsimplex{d-1}$. For $v\in\nsimplex{d-1}$ and $i\in[d]$ denote by $d_i(v)$ the distance of $v$ to the $i$-th vertex of $\nsimplex{d-1}$. Then each sector is determined by the permutation of $[d]$ induced by ordering the $d_i(v)$  increasingly. 
The sectors are dual to the vertices of the $d$-dimensional {permutahedron}.
See Figure \ref{fig:affine_sectors} for an illustration.

\begin{figure}[h]
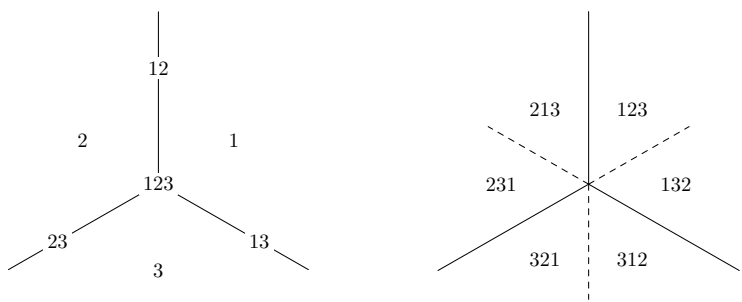

\centering
\includegraphics[width=4cm]{affine_php.1}
\qquad\qquad
\includegraphics[width=4cm]{affine_php.2}
\caption[The arrangement of reflection hyperplanes from a tropical line.]{A $2$-dimensional \thp with its types (on the left) and the corresponding arrangement $\mc F$ of  hyperplanes (on the right). Moreover, the bijection between the open sectors of $\mc F$ and the permutations of $[3]$ is given.}
\label{fig:affine_sectors}
\end{figure}

For $X\subseteq\{I\mid\emptyset\neq I\subset[d]\}$ we say that $A\subseteq\T^{d-1}$  \defn[approximate]{approximates} $T_X\coloneq\bigcup_{I\in X}T_I$ if:
\begin{itemize}
\item For each $I\in X$, there is $\eps_I>0$ such that  $T_I$ is contained in $A$ except possibly for an $\eps_I$-neighbourhood of the (relative) boundary $\bnd T_I$.

\item For each $I\nin X$ there is $\eps_I>0$ such that $T_I\cap A$ is contained in an $\eps_I$-neighbourhood of $\bnd T_I$.
\end{itemize}

Intuitively, the set $A$ is supposed to contain ``almost everything'' of $T_I$ if $I\in X$ and ``almost nothing'' of $T_I$ if $I\nin X$. Then $T_X$ is homeomorphic to  $A$.
We will be interested in approximating neighbourhoods for $X=\{a,b,a\cup b\}$ with $a,b\subset[d]$.
See Figure \ref{fig:ex_repr_set} for an illustration.

\begin{figure}[t]
\centering
\includegraphics[width=3cm]{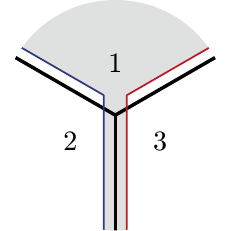}
\qquad\qquad
\includegraphics[width=3cm]{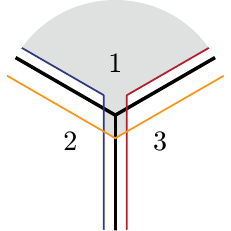}
\caption{Approximating neighbourhoods corresponding to $a=1$ and $b=23$ (on the left), respectively $a=1$ and $b=123$ (on the right).}
\label{fig:ex_repr_set}
\end{figure}

For $I\subseteq[d]$ and $p\in\T^{d-1}$ denote by $\contypes Ip$ the set of all types that are approximated by $\shifths Ip$. 
The following lemma characterises the types in $\contypes  Ip$:
\begin{lemma} Let $H\subset\T^{d-1}$ be a \thp and  $\pi=(\pi_1,\ldots,\pi_d)$ the permutation of $[d]$ corresponding to a point $p\in\T^{d-1}\setminus(\bigcup \mc F)$.
\begin{enumerate}
\item \label{it:approx_intersection}Let $\emptyset\neq J\subset [d]$. Then
$T_J\cap H_{i,p}^+\neq\emptyset$ \tiff each $j\in J\setminus\{i\}$ comes before $i$ in $\pi$. 
\item
$\contypes Ip$ only depends on the open sector of $\mc F$ in which $p$ lies, hence on the permutation corresponding to $p$.
\item
Let $i\in[d]$ and $J\subseteq[d]$. Then $J\in\contypes ip$ \tiff $i\in J$ and $\{i\}\cup\{q\mid q \text{ comes before }i\text{ in }\pi\}\supseteq J.$

\item
$\contypes Ip=\bigcup\limits_{i\in I}\contypes ip$.
\end{enumerate}
\end{lemma}

\begin{proof}$ $
\begin{enumerate}
\item 
We first prove the statement for $\card J=1$. So assume $J=\{j\}$. But then it is easy to see that $H_{i,p}^+$ intersects $T_j=H_j^+$ \tiff $j=i$ or $j$ comes before $i$ in $\pi$.

The general statement (for $\card J\geq 2$) follows since intersections of tropically convex sets are tropically convex and $H_{i,p}^+$ is open.

\item This is clear.

\item
Assume \twlog that $i=1$.
We will prove the statement by induction over the length of $\pi$, \ie the minimal number of transpositions needed to write $\pi$ as a product of transpositions.

It is clear that the only type approximated by $H_{1,p}^+$ for $\pi_1=1$ is $\{1\}$.

\medskip
Now assume the statement is true for $\pi=(\pi_1,\ldots,\pi_d)$ and apply one transposition $\tau=(\pi_j,\pi_{j+1})$ with $\pi_j<\pi_{j+1}$ to obtain $\pi'$; \ie $\tau$ swaps two neighbouring entries of $\pi$, increasing the length by one. Denote by $p'$ one point in the $\pi'$-sector of $H$. In particular, we can always choose $p'$ such that $H_{i,p'}^+\supset H_{i,p}^+$.

This means we move $p$ into a neighbouring sector of $\mc F$. There are two cases:
\begin{itemize}
\item If both $\pi_j,\pi_{j+1}$ come before or after $1$ in $\pi$, the types approximated by $H_{i,p}^+$ do not change. Indeed, we can decrease the length by one by  relabeling  the sectors $\pi_j\leftrightarrow \pi_{j+1}$.

\item Assume $\pi_j=1$. By passing from sector $\pi$ to the sector $\pi'$ we cross the hyperplane $\lin T_{1,\pi_{j+1}}$. 
We now show that then $\contypes i{p'}=\contypes ip \cup\:\{r\cup\{\pi_{j+1}\}\mid r\in \contypes ip\}$.

So let $r\in\contypes ip$ and denote $r'\coloneq r\cup\{\pi_{j+1}\}$. \Ie $T_r$ is approximating by $H_{i,p}^+$. But then clearly $T_r$ is also approximated by $H_{i,p'}^+\supset H_{i,p}^+$. Moreover, $T_{r'}$ is approximated by $H_{i,p'}^+$ since it intersects $H_{i,p'}$ and is contained in the boundary of $T_r$.
That $\contypes i{p'}$ is not larger than this follows from (\ref{it:approx_intersection}).
\end{itemize}

\item 
This follows from $\shifths Ip=\bigcup\limits_{i\in I} \shifths ip$. 
\qedhere
\end{enumerate}
\end{proof}

\begin{lemma}\label{lem:aff_php_elim}
Let $H$ be a \tphp in $\T^{d-1}$ and $\emptyset\neq I,J\subset[d]$.
Then we can represent an approximating neighbourhood of $T_I\cup T_J\cup T_{I\cup J}$ as an intersection of affine pseudohalfspaces.
\end{lemma}
\begin{proof}
It suffices to prove the statement for usual \thps since the PL-homeomorphism taking a \thp to a \tphp also maps our affine pseudohalfspaces in an appropriate way. See Figure \ref{fig:ex_repr_set} for an example.

It suffices to show that for each set $K\neq I,J,I\cup J$ there are $L\subset[d]$ and $p\in\T^{d-1}$ such that $\contypes Lp$ contains $I,J,I\cup J$ but not $K$.
Then we only need to intersect all of these affine pseudohalfspaces for each $K\neq I,J,I\cup J$.

Note that  the open sectors of the arrangement $\mc F$ of linear hyperplanes (and hence the points $p\in\T^{d-1}\setminus(\bigcup\mc F)$) correspond to permutations in $\pi\in\Sym_d$. See again Figure \ref{fig:affine_sectors}.
\begin{itemize}
\item First assume that there is $x\in K\setminus(I\cup J)$. Then we can choose $\pi$ to end in $x$ to make sure $x$ will never occur in any element of $\contypes Lp$. In detail, choose $i\in I, j\in J$. Let $L=\{i,j\}$ and let $p$ be such that $i'\leq i$ for each $i'\in I$ and $j'\leq j$ for each $j\in J$ and $x>y$ for each $y\neq x$.
 
\item If $K\subseteq I\cap J$, choose $i\in I\setminus K,j\in J\setminus K$, which exist since $K\neq I,J$. Let $L=\{i,j\}$ and choose $\pi$ in such a way that the elements of $I-\{i\}$ and $J-\{j\}$ come first. Then any element of $\contypes Lp$ contains either $i$ or $j$. Thus, $K\nin \contypes Lp$ and it is easy to check that $I,J,I\cup J\in\contypes Lp$. 

\item Otherwise there is $i\in (I\cup J)\setminus K$. Let $L=\{i\}$ and let $\pi$ begin with the elements of $(I\cup J)-\{i\}$. Then every element of $\contypes Lp$ contains $i$. Hence $K\nin\contypes Lp$. Again, it is easy to see that $I,J,I\cup J\in\contypes Lp$.\qedhere
\end{itemize}
\end{proof}
See Figure \ref{fig:approx_blow_up} for an example.

\begin{lemma}
Let $H$ be a \tphp with apex $0$ in $\T^{d-1}$. For each $(I,\pi)$ with $\emptyset\neq I\subset[d]$ and $\pi\in\Sym_d$ fix one point $p_{I\pi}$ in the $\pi$-sector of $H$ in such a way that the arrangement of tropical hyperplanes with apices in $\{0\}\cup\{p_{I\pi}\}$ is in general position. Then \[\mc H\coloneq\{H_{I,p_{I\pi}}\mid\emptyset\neq I\subset[d], \pi\in\Sym_d\}\] is an arrangement of affine \phps.
\end{lemma}

\begin{proof}
This follows by applying the Topological Representation Theorem \ref{thm:toprep2} to realisable \toms.
\end{proof}

We can extend the above construction to tropical \emph{pseudo}hyperplanes as follows: Let $H$ be a \tphp. Then $H$ is the image of a \thp $H'$ under a PL-homeomorphism $\phi$ of $\T^{d-1}$. Then we define $H_{I,p}\coloneq \phi (H'_{I,p})$. 

Note that by continuity of $\phi$ and the fact that $\phi$ fixes the boundary of $\T^{d-1}$ we can always choose the point $p$ so that $H_{I,p}$ lies very close to $H_I$. 
Now consider an arrangement $\mc A=(H_i)_{i\in[n]}$ of \tphps. We can do the above construction for each of them individually.

\medskip
If $H$ is a \tphp in such an arrangement, then we can  consider $H_{I,p}$ as $H'_I$ for the new hyperplane $H'$ that arises by blowing up $H$ with respect to the permutation $p$.
The following is immediate:
\begin{lemma}
Let $S=\nsimplex{d-1}$ be the \mixsd dual to a tropical hyperplane $H$ and fix $\pi\in\Sym_d$. Then the blow-up of $H$ with respect to $\pi$ corresponds to adding a second tropical hyperplane with apex in the $\overline\pi$-sector of $H$.
\end{lemma}

\begin{figure}
\centering
\includegraphics[height=3cm]{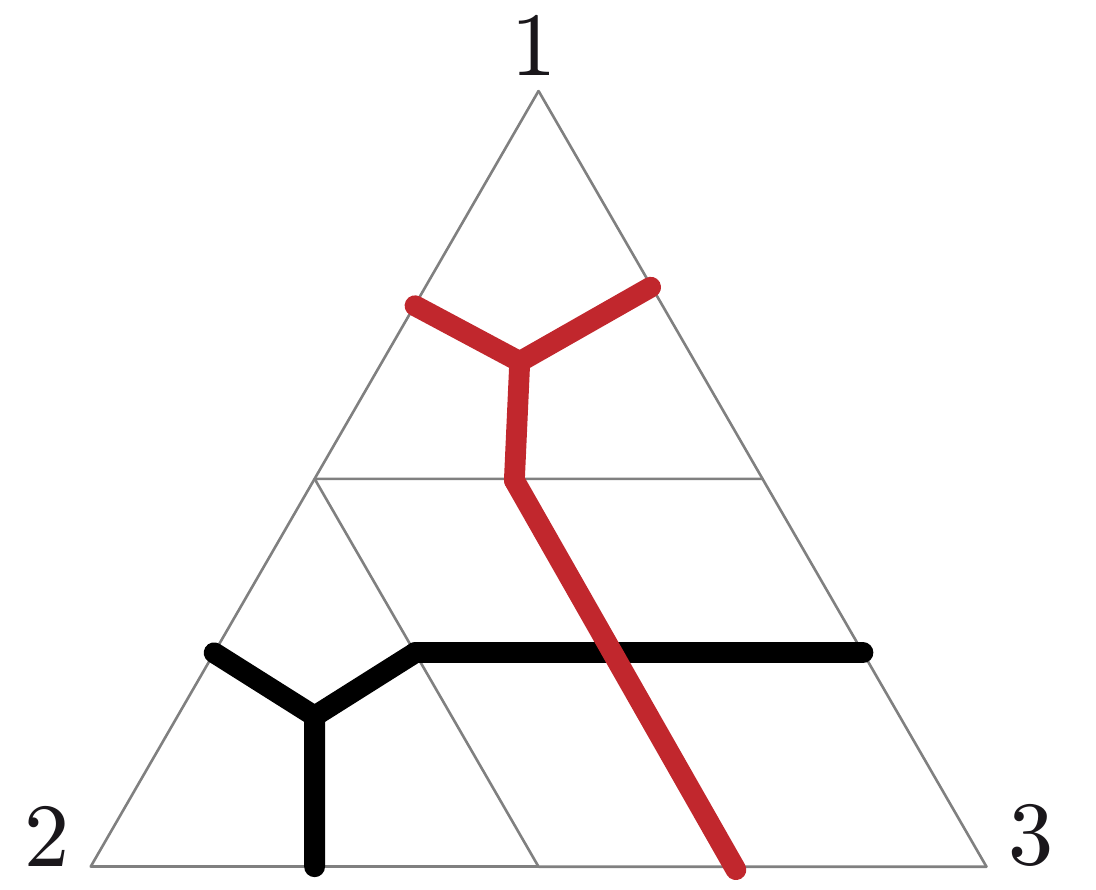}\qquad\qquad\includegraphics[height=3cm]{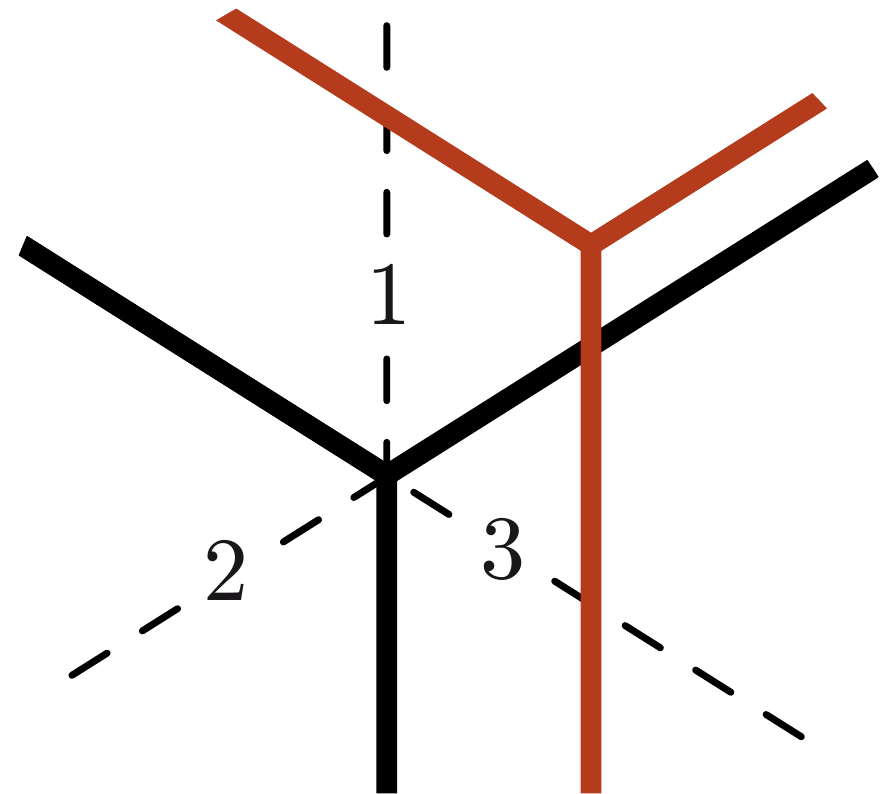}
\caption[The blow-up as an $n$-placing extension.]{The blow-up of the black \tphp with respect to $\pi=(2,3,1)$ yields a new \tphp with apex in the $(1,3,2)$-sector of the first \tphp.}
\label{fig:blow_up_perm}
\end{figure}
See Figure \ref{fig:blow_up_perm} for an illustration.

We can use blow-ups to construct an affine pseudohyperplane arrangement $\mc H$ for a given \tphp $H$. For each $(I,\pi)$ with $\pi\in\Sym_d$ and $\emptyset\neq I\subset [d]$ perform one blow-up of $H$ with respect to $\overline\pi$ and denote the \tphp emerging from this blow-up by $H^{I,\pi}$.

We then obtain $(2^d-2)d!$ new tropical hyperplanes (one for each $(I,\pi)$ and hence a \mixsd of $\dilsimp{((2^d-2)d!+1)}{d-1}$. With this we can, in the dual arrangement of \tphps, define $\mc H =H^{I,\pi}_I$.

See Figure \ref{fig:approx_blow_up} for an illustration.

\begin{figure}
\centering
\includegraphics[height=3cm]{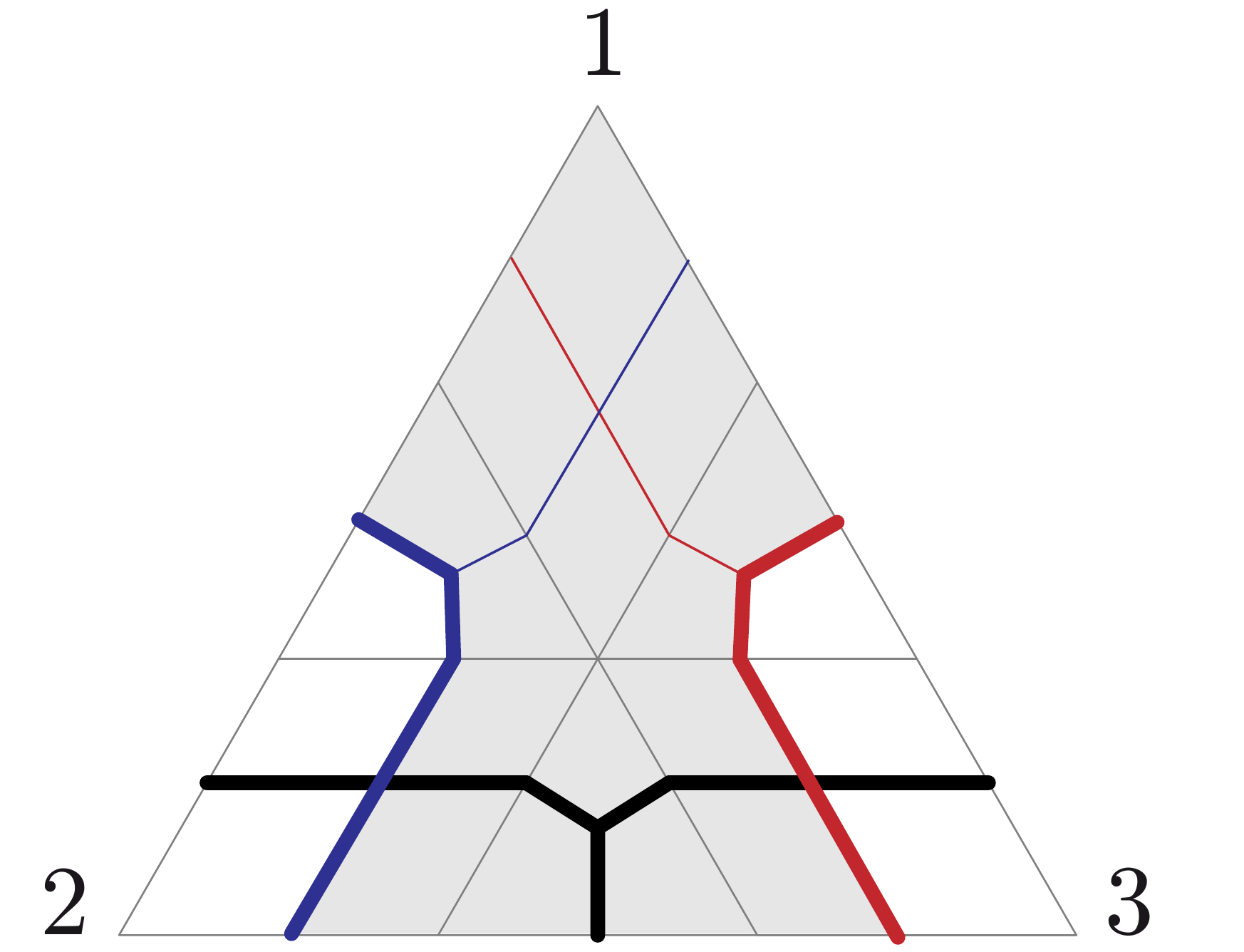}
\vspace{-2mm}
\caption{An approximating neighbourhood for $a=1,b=23$ as an intersection of affine pseudohalfspaces in a blow-up of the black \tphp.}
\label{fig:approx_blow_up}
\end{figure}

\begin{theorem}\label{thm:tphpa_elim}
The types in a \tphpa as in Definition \ref{def:tphpa2} satisfy the elimination axiom of a \tom.
\end{theorem}
\begin{proof}
Let $A,B$ be types in a \tphpa $\mc T$.
By Proposition \ref{prop:elim_iff_connected} it suffices to show that $S_{AB}$ is connected.

By Lemma \ref{lem:aff_php_elim} we can approximate the set $S_{AB}=\{C\mid C_i\in \{A_i, B_i,A_i\cup B_i\}\}$ as an intersection $X=\bigcap H_i^+$ of pseudohalfspaces in an arrangement of affine pseudohyperplanes obtained by suitable blow-ups of $S$.
By Proposition \ref{prop:intersection_of_aff_phss_connected}, $X$ is connected.

Moreover, $S_{AB}$ is homotopic to $X$. To see this we shrink the new \tphps, that were added during the blow-ups. Denote by $S'$ the blow-up of $S$ and assume \twlog that the original $n$ \tphps have indices $1,\ldots,n$. Moreover, assume that $S'$ is a \mixsd of $\dilsimp N{d-1}$. Consider the following {homotopy}:
\[\begin{split}H:&~[0,1]\times S'\to S\\ 
&\left(\lambda, \sum_{i=1}^N C_i \right)\mapsto \sum_{i=1}^n C_i+ (1-\lambda)\sum_{i=n+1}^N C_i.
\end{split}\]
It is clear that $H$ is continuous.
Moreover, $H(X,0)=X$ and $H(X,1)=S_{AB}$ and hence $S_{AB}$ is homotopic to $X$.
\end{proof}

\subsection{Non-fine mixed subdivisions}\label{sec:elimination_nongen}

In this section we prove that arbitrary \dnmixsds satisfy the elimination property.

We can still construct approximating neighbourhoods by means of blowing up \tphps even if the \mixsd is not fine.

The following is clear from the above:

\begin{lemma}
Let $S$ be a (not necessarily fine) \dnmixsd. Then $\bigcup \{\mc H_i\}$ is an arrangement of affine \phps.
\end{lemma}

With this we are now ready to prove the main result of this chapter:
\begin{theorem}\label{thm:elimination_nongen}
Every \dnmixsd satisfies the elimination property.
\end{theorem}
\begin{proof}

If we repeatedly blow-up $S$ with respect to any $(i,\pi,I)$ we obtain $n(2^d-2)d!$ new \tphps (one for each $(i,\pi,I)$) and hence a \mixsd of $\dilsimp{(n+n(2^d-2)d!)}{d-1}$, in which we find our $H_{I,p}$s. It remains to show that these again form an arrangement of affine pseudohyperplanes. But this follows since if we delete the $n$ original \tphps we obtain a \tphpa in general position.

From here on the proof works as for Theorem \ref{thm:tphpa_elim}:
\end{proof}

From this we immediately obtain the following corollaries:
\begin{corollary}[{\cite[Conjecture 5.1]{Ardila/Develin}}]
Tropical oriented matroids with parameters $(n,d)$ are in one-to-one correspondence with \dnmixsds and subdivisions of $\simpprod{n-1}{d-1}$.
\end{corollary}

This completes the proof of the equivalence of the five concepts of \toms, \tphpas I (\cite[Definition 4.3]{toprep1}) and II (Definition \ref{def:tphpa2}), \dnmixsds and subdivisions of $\simpprod{n-1}{d-1}$ depicted in Figure \ref{fig:big_plan}.

\bigskip
Moreover, the duality relation between \dnmixsds and $\dilsimp d{n-1}$ implies that the dual of a \tom is itself a \tom.
\begin{corollary}[{\cite[Conjecture 5.5]{Ardila/Develin}}]\label{cor:dual_tom}
The dual \index{dual!of a tropical oriented matroid} of a \tom with parameters $(n,d)$ is a \tom with parameters $(d,n)$.
\end{corollary}

\nocite{polymake}
\nocite{mydiss}
\printbibliography
\end{document}